\documentclass[a4paper,10pt]{article}
\usepackage{amsfonts}
\usepackage{amsthm}
\usepackage{amsmath}
\usepackage{amscd}
\usepackage[latin2]{inputenc}
\usepackage{t1enc}
\usepackage[mathscr]{eucal}
\usepackage{indentfirst}
\usepackage{graphicx}
\usepackage{graphics}
\usepackage{pict2e}
\usepackage{epic}
\numberwithin{equation}{section}
\usepackage[left=1.25in,right=1.25in,top=1in]{geometry}
\usepackage{epstopdf}
\usepackage{amsmath,amsthm,verbatim,amssymb,amsfonts,amscd, graphicx}
\usepackage{mathtools}
\usepackage[dvipsnames]{xcolor}
\usepackage{authblk}

\usepackage[backend=bibtex,firstinits=true,maxbibnames=9,natbib=true,url=false,sorting=nyt,doi=true,backref=false]{biblatex}

\renewbibmacro{in:}{\ifentrytype{article}{}{\printtext{\bibstring{in}\intitlepunct}}}

\DeclareMathOperator{\sgn}{sgn}

\usepackage[colorlinks,linkcolor = red, citecolor=blue]{hyperref}

\definecolor{mypink1}{rgb}{0.858, 0.188, 0.478}
\definecolor{mypink2}{RGB}{219, 48, 122}
\definecolor{mypink3}{cmyk}{0, 0.7808, 0.4429, 0.1412}
\definecolor{mygray}{gray}{0.6}

\newlength{\leftstackrelawd}
\newlength{\leftstackrelbwd}
\def\leftstackrel#1#2{\settowidth{\leftstackrelawd}%
{${{}^{#1}}$}\settowidth{\leftstackrelbwd}{$#2$}%
\addtolength{\leftstackrelawd}{-\leftstackrelbwd}%
\leavevmode\ifthenelse{\lengthtest{\leftstackrelawd>0pt}}%
{\kern-.5\leftstackrelawd}{}\mathrel{\mathop{#2}\limits^{#1}}}

 \theoremstyle{plain}
\newtheorem{Th}{Theorem}[section]

\newtheorem{Cor}[Th]{Corollary}
\newtheorem{Prop}[Th]{Proposition}
\theoremstyle{definition}
\newtheorem{Def}[Th]{Definition}

\newtheorem{Rem}[Th]{Remark}
\newtheorem{?}[Th]{Problem}

\newcommand{\diam}{{\rm{diam}}}

\def\R{{\mathbb R}}

\def\N{{\mathbb N}}

\def\E{{\mathbb E}}
\def\R{{\mathbb R}}

\def\P{{\mathbb P}}

\def\A{{\mathbf A}}
\def\x{{\mathbf x}}
\def\b{{\mathbf b}}
\def\a{{\mathbf a}}

\def\h{{\mathbf{h}}}

\def\s{{\sigma}}
\def\e{{\varepsilon}}

\def\v{{\mathbf{v}}}

\def\y{{\mathbf{y}}}
\def\g{{\mathbf{g}}}
\def\u{{\mathbf{u}}}
\def\w{{\mathbf{w}}}

\begin{document}
\title{Analysis of The Ratio of $\ell_1$ and $\ell_2$ Norms in Compressed Sensing}
\date{}

\author[1]{\small Yiming Xu}
\author[1,2]{\small Akil Narayan}
\author[3]{\small Hoang Tran}
\author[4,5]{\small Clayton G. Webster}
\affil[1]{\small Department of Mathematics, University of Utah, Salt Lake City, 84112\\
       \texttt{yxu@math.utah.edu}}
\affil[2]{\small Scientific Computing and Imaging Institute, University of Utah, Salt Lake City, 84112\\
       \texttt{akil@sci.utah.edu}}
\affil[3]{\small Computer Science and Mathematics Division, Oak Ridge National Laboratory, Oak Ridge, TN, 37831\\
 \texttt{tranha@ornl.gov}}
\affil[4]{\small Oden Institute for Computational Engineering \& Sciences, The University of Texas at Austin, Austin, TX, 78712\\
 \texttt{claytongwebster@utexas.edu}}
\affil[5]{\small Behavioral Reinforcement Learning Lab, Lirio LLC., Knoxville, TN, 37923}

\renewcommand\Authands{ and }
  
\maketitle

\begin{abstract}
 We study the ratio of $\ell_1$ and $\ell_2$ norms ($\ell_1/\ell_2$) as a sparsity-promoting objective in compressed sensing. 
 We first propose a novel criterion that guarantees that an $s$-sparse signal is the local minimizer of the $\ell_1/\ell_2$ objective; our criterion is interpretable and useful in practice. We also give the first uniform recovery condition using a geometric characterization of the null space of the measurement matrix, and show that this condition is easily satisfied for a class of random matrices. We also present analysis on the \textcolor{black}{robustness} of the procedure when noise pollutes data. Numerical experiments are provided that compare $\ell_1/\ell_2$ with some other popular non-convex methods in compressed sensing. Finally, we propose a novel initialization approach to accelerate the numerical optimization procedure. We call this initialization approach \emph{support selection}, and we demonstrate that it empirically improves the performance of existing $\ell_1/\ell_2$ algorithms.\end{abstract}


\section{Introduction}

The goal of compressed sensing (CS) problem is to seek the sparsest solution of an underdetermined linear system:
\begin{align}
\min \|\x\|_0 \ \ \text{subject to}\ \A\x=\b,\label{cs}
\end{align}
where $\x\in\R^n, \b\in\R^m$ and $\A\in\R^{m\times n}$ with $m\ll n$. The quasinorm $\| \x\|_0$ measures the number of nonzero components in $\x$. In CS applications, one typically considers $\x$ as the frame/basis coordinates of an unknown signal, and it is typically assumed that the coordinate representation is \textit{sparse}, i.e., that $\| \x\|_0$ is ``small''. 
$\A$ is the measurement matrix that encodes linear measurements of the signal $\x$, and $\b$ contains the corresponding measured values. In the language of signal processing, (\ref{cs}) is equivalent to applying the sparsity decoder to reconstruct a signal from the undersampled measurement pair ($\A$, $\b$). A naive, empirical counting argument suggests that if $m \ll n$ measurements $\b$ of an unknown signal $\x$ are available, then we can perhaps compute the original signal coordinates $\x$, assuming $\x$ is approximately $m$-sparse. The optimization \eqref{cs} is the quantitative manifestation of this argument.

It was established in \cite{Donoho_2003} that under mild conditions, (\ref{cs}) has a unique minimizer. In the rest of the paper we assume that the minimizer is unique and denote it by $\x_0$. One of the central problems in compressed sensing is to design an effective algorithm to find $\x_0$: Directly solving (\ref{cs}) via combinatorial search is NP-hard \cite{Natarajan_1995}. A more practical approach, which was proposed in the seminal work \cite{Donoho_2006}, is to relax the sparsity measure $\| \cdot\|_0$ to the convex $\ell_1$ norm $\| \cdot\|_1$:
\begin{align}
\min \|\x\|_1 \ \ \text{subject to}\ \A\x=\b.\label{l1}
\end{align} 

The convexity of the problem \eqref{l1} ensures that efficient algorithms can be leveraged to compute solutions. Many pioneering works in compressed sensing have focused on understanding the equivalence between (\ref{cs}) and (\ref{l1}), see \cite{Donoho_2006,Candes_2006,Donoho_2001}. 
A major theoretical cornerstone of such equivalence results is the \emph{Null Space Property} (NSP), which was first introduced in \cite{Cohen_2008} and plays a crucial role in establishing sufficient and necessary conditions for the equivalence between (\ref{cs}) and (\ref{l1}). \textcolor{black}{A sufficient condition for such an equivalence is called an \textit{exact recovery condition}.} A closely related but stronger condition is the \emph{Restricted Isometry Property} (RIP), see \cite{Candes_2006}. The RIP is more flexible than the NSP for practical usage,
yet conditions given by both the NSP and RIP are hard to verify in the case when measurements (i.e., the matrix $\A$) are deterministically sampled. An alternative approach based on analyzing the \emph{mutual coherence} of $\A$ produces a practically computable but suboptimal condition, see \cite{Donoho_2003}. We will use a slightly more general definition of the NSP that was introduced in \cite{foucart2009sparsest}:

\begin{Def}[Null Space Property]
  Given $s \in \N$ and $c > 0$, a matrix $\A\in\R^{m\times n}$ satisfies the $(s,c)$-NSP in the quasi-norm $\ell_q$ ($0<q\leq 1$) if for all $\h\in\ker(\A)$ and $T \in [n]_s$, we have 
\begin{align}
\|\h_T\|_q^q<c\|\h_{T^\complement}\|_q^q. \label{NSP}
\end{align}
Here, $[n]_s$ is the collection of all subsets of $\{1,\ldots n\}$ with cardinality at most $s$,
\begin{align*}
  [n]_s &\coloneqq \left\{ T \subset [n]\;\; \big|\;\; \left|T\right| \leq s \right\}, & 
  [n] &\coloneqq \{1, \ldots, n\},
\end{align*}
$\h_T$ is the restriction of $\h$ to the index set $T$, and $T^\complement \coloneqq [n] \backslash T$.
\end{Def}

Nearly all exact recovery conditions based on the RIP are probabilistic in nature. \textcolor{black}{This means that such analysis typically is split into two major thrusts: (i) the first one establishes that (\ref{cs}) and (\ref{l1}) are equivalent for a class of sparse signals if $\A$ satisfies an RIP condition, and (ii) the second one proves that the RIP condition for a suitable random matrix $\A$ is achievable with high probability. }Such random arguments appear to be necessary in practice for the RIP analysis in order to mitigate pathological measurement configurations. 

Under proper randomness assumptions, an alternative approach that circumvents the RIP also yields fruitful results in the study of (\ref{l1}), see \cite{Zhang_2013,Vershynin_2015}. This approach is more reliant on a geometric characterization of the nullspace of the measurements, and therefore could be potentially adapted to analyzing non-convex objectives with similar geometric interpretations. We take this approach for analysis in this paper.

Although (\ref{l1}) has attracted a lot of interest in the past decades, the community realized that $\ell_1$ minimization is not as robust for computing sparsity-promoting solutions compared to other objective functions, in particular compared to other non-convex objectives. This motivates the study of non-convex relaxation methods (use non-convex objectives to approximate $\| \cdot\|_0$), which are believed to be more sparsity-aware. Many non-convex objective functions, such as $\ell_q$ ($0<q<1$) \cite{Gribonval_2007,Chartrand_2007,foucart2009sparsest}), reweighted $\ell_1$ \cite{Candes_2008}, CoSaMP \cite{Needell_2010}, IHT \cite{blumensath2009iterative}, $\ell_1-\ell_2$ \cite{Yin_2014,Yin_2015}, and $\ell_1/\ell_2$ \cite{Hoyer,Hurley_2008,Yin_2014,Rahimi_2019}, are empirically shown to outperform $\ell_1$ in certain contexts. However, relatively few such approaches have been successfully analyzed for theoretical recovery. In fact, obtaining exact recovery conditions and robustness analysis for general non-convex optimization problems is difficult, unless the objective possesses certain exploitable structure, see \cite{Lv_2009,Tran_2019}. 

We aim to investigate exact recovery conditions as well as the robustness for the objective $\ell_1/\ell_2$ in this paper. We are interested in providing conditions under which \eqref{cs} is equivalent to the following problem:
\begin{align}
\min \frac{\|\x\|_1}{\|\x\|_2} \ \ \text{subject to}\ \A\x=\b.\label{l1/l2}
\end{align}
To our knowledge, $\ell_1/\ell_2$ does not belong to any particular class of non-convex functions that has a systematic theory behind it. This is mainly because $\ell_1/\ell_2$ is neither convex nor concave, and is not even globally continuous. However, there are a few observations that make this non-convex objective worth investigating. First, in \cite{Petrosyan_2019} it was shown numerically that the $\ell_1/\ell_2$ outperforms $\ell_1$ by a notable margin in jointly sparse recovery problems (in the sense that many fewer measurements are required to achieve the same recovery rate); particularly, $\ell_1/\ell_2$ admits a high-dimensional generalization called \emph{orthogonal factor} and the corresponding minimization problem can be effectively solved using modern methods of manifold optimization. Understanding $\ell_1/\ell_2$ in one dimension would offer a baseline for its higher-dimensional counterparts. Secondly, in the matrix completion problem \cite{Cand_s_2012}, one desires a matrix with minimal rank under the component constraints. Note that the rank of a matrix is the $\ell_0$ measure of its singular value vector. A natural relaxation of rank to a more regular objectives include the so-called numerical intrinsic rank, which is defined by the ratio $\ell_1/\ell_\infty$ of the singular value vector, and the numerical/stable rank, which is defined by the ratio $\ell_2/\ell_\infty$ of the singular value vector. This suggests that the ratio between different norms might be a useful function to measure sparsity (complexity) of an object, and therefore leads us to study the objective $\ell_1/\ell_2$ in compressed sensing. 

\textcolor{black}{A few attempts have been made recently to reveal both the theory and applications behind the $\ell_1/\ell_2$ problem \cite{Krishnan_2011, Esser_2013, Yin_2014, Rahimi_2019, wang2020limited}. However, the existing analysis is either applicable only for non-negative signals, or yields a local optimality condition which is often too strict in practice. The investigation of efficient algorithms for solving the $\ell_1/\ell_2$ minimization is also an active area of research \cite{Rahimi_2019,Wang_2020, boct2020extrapolated, zeng2020analysis, taominimization, wang2021minimizing}.}

Our contributions in this paper are two-fold. First we propose a new local optimality criteria which provides evidence that a large ``dynamic range'' may lead to better performance of an $\ell_1/\ell_2$ procedure, as was observed in \cite{Wang_2020}. We also conduct a first attempt at analyzing the exact recovery condition (global optimality) of $\ell_1/\ell_2$; a sufficient condition for uniform recoverability as well as some analysis of the robustness to noise are also given. We also provide numerical demonstrations, in which a novel initialization step for the optimization is proposed and explored to improve the performance {of} existing algorithms. We remark that since this problem is non-convex, none of the results in this paper are tight; they only serve as the initial insight into certain aspects of the method that have been observed in practice. 

The rest of the paper is organized as follows. In Section \ref{sec:cs} we briefly introduce the results in \cite{Zhang_2013} and \cite{Vershynin_2015} obtained by a high-dimensional geometry approach, which are relevant to our analysis. In Section \ref{local optimality}, we give a new local optimality condition ensuring that an $s$-sparse vector is the local minimizer of the $\ell_1/\ell_2$ problem. In Section \ref{exact recovery} we investigate the uniform recoverability of $\ell_1/\ell_2$ and propose a new sufficient condition for this recoverability. We also show that this condition is easily satisfied for a large class of random matrices. In Section \ref{stability}, we give some elementary analysis on how the solution to $\ell_1/\ell_2$ minimization problem is affected in the presence of noise. In Section \ref{numerical experiments}, we provide some numerical experiments to support our findings and propose a novel initialization technique that further improves an existing $\ell_1/\ell_2$ algorithm from \cite{Wang_2020}. In Section \ref{sec:con},  we summarize our findings as well as point out some possible directions for future investigation.

\section{A geometric perspective on $\ell_1$ minimization}\label{sec:cs}

Geometric interpretation of compressed sensing first appeared in an abstract formulation of the problem in Donoho's original work \cite{Donoho_2006}.  In this section, we will take a selection of geometric views on $\ell_1$ minimization based on the discussions in \cite{Vershynin_2015} and \cite{Zhang_2013}, which do not hinge on RIP analysis. We will see that they provide valuable insight for our analysis of (\ref{cs}) in the case of non-convex relaxation.

To interpret (\ref{l1}) geometrically, we assume that entries of $\A$ are iid standard normal, i.e., $(\A)_{i,j} \sim \mathcal{N}(0,1)$. 
In this case, an $\x$ solving \eqref{l1} belongs to the translate of a subspace that is uniformly drawn from the Grassmanian $G_{n-m, n}$, where, 
\begin{align*}
  G_{r,n} = \left\{ A \subset \R^n \;\big|\; A \textrm{ is an $r$ dimensional subspace} \right\}.
\end{align*}
The objective function, on the other hand, can be considered as an \textcolor{black}{origin-centered} symmetric convex body, i.e., a scaled $\ell_1$ ball in $\R^n$. Therefore, (\ref{l1}) is associated to a problem of understanding the random section of a convex set in $\R^n$. We thus seek to understand random sections of convex sets.

We began by introducing the approach in \cite{Vershynin_2015}. Visualization of convex sets in high dimensions often depends on two parts: the bulk and the outliers. The bulk is the largest inscribed part of a convex set that resembles an ellipsoid, and the outliers are those points outside the bulk contributing to the diameter of the set. As the name outliers suggests, a random low-dimensional section of a convex set tends to avoid outliers, and the resulting shape is close to the section of the bulk, i.e., an ellipsoid. As the dimension increases, the random section is more likely to capture the outliers, and the diameter of the intersected region will grow in a manner determined by the geometry of the convex set. This is made precise by the following theorems:


\begin{Th}\cite[\textcolor{black}{Theorem 3.3, Dvoretzky}]{Vershynin_2015}\label{T1}
  Let $0<\e, \delta<1$ be two fixed numbers and $d\in\N$. Let $K$ be an \textcolor{black}{origin-centered} symmetric convex body in $\R^n$ such that the largest ellipsoid inscribed in it is the unit Euclidean ball. Let $E$ be a random subspace drawn uniformly from $G_{d,n}$. There exists $R>0$ which only depends on $K$ such that with probability at least $1-\delta$, 
\begin{align}
(1-\e)B(R)\subset K\cap E\subset (1+\e)B(R),
\end{align}
provided that $d\leq C(\e, \delta)\log n$, where $B(R)$ is the Euclidean ball of radius $R$ in $E$ and $C(\e, \delta)$ is a constant depending only on $\e$ and $\delta$. The condition on $d$ can be improved to $d\leq C(\e, \delta)n$ when $K$ is the $\ell_1$ ball with radius $\sqrt{n}$ and $R=1$.
\end{Th}

\begin{Th}\cite[\textcolor{black}{Theorem 5.1, $M^*$ bound}]{Vershynin_2015}\label{T2}
Let $K$ be a bounded subset of $\R^n$. Let $E$ be a random subspace drawn uniformly from $G_{n-m,n}$. Then, 
\begin{align}
\E\sup_{\u\in K\cap E}\|\u\|_2\leq \sqrt{\frac{8\pi}{m}}\cdot\E\sup_{\u\in K}|\langle \g,\u\rangle|,\label{5}
\end{align}
where $\g\sim\mathcal{N}(\mathbf{0},\mathbf{I}_n)$. 
\end{Th}

\textcolor{black}{Theorem \ref{T1} and \ref{T2} describe the low-dimensional and high-dimensional section of an origin-centered symmetric convex body, respectively. Their proofs can be found in \cite{Vershynin_2018}, and the idea of the latter is crucial for the practicality part of our result.} The quantity $\E\sup_{\u\in K}|\langle \g,\u\rangle|$ on the right-hand side of (\ref{5}) is closely related to the concept of Gaussian width or Gaussian complexity of a set $K$:

\begin{Def}[Gaussian complexity]\label{Gaussian complexity}
  Let $K$ be a bounded set in $\R^n$. The Gaussian width of $K$ is defined by $w(K)=\E\sup_{\x\in K}\langle\g, \x\rangle$, where $\g\sim\mathcal{N}(\mathbf{0},\mathbf{I}_n)$. The Gaussian complexity of $K$ is defined as $w'(K)=w(K-K)$, where $K - K$ is the Minkowski difference between $K$ and itself.
\end{Def}

It is easy to check that the Gaussian width of a set remains unchanged after taking the convex hull, i.e., $w(K)=w(\text{conv}(K))$, implying an immediate upper bound for the Gaussian width of the unit $\ell_1$ ball $B^n_1$ in $\R^n$ ($B_1^n$ is the convex hull of the set $\{\pm\mathbf{e}_i, i\leq n\}$, where $\mathbf{e}_i$ is the $i$-th unit vector in $\R^n$):
\begin{align}
w(B^n_1)=\E\max_{i\leq n}|(\g)_i|\leq \sqrt{8\log n}, \label{l1-ball:gauss-width}
\end{align} 
where $\g\sim\mathcal{N}(\mathbf{0}, \mathbf{I}_n)$.

Indeed, when $K$ is a symmetric convex body centered at origin, then $K - K = 2K$, so that (\ref{5}) implies
\begin{align}
\E\diam(K\cap E)= 2\E\sup_{\u\in K\cap E}\|\u\|_2\leq \sqrt{\frac{8\pi}{m}}\cdot 2\E\sup_{\u\in K}\langle \g,\u\rangle=\sqrt{\frac{8\pi}{m}}\cdot w'(K). \label{6}
\end{align}   
Note that $\sup_{u\in K-K}\langle \g,\u\rangle$ is the distance between two hyperplanes (with normal direction $\g$) that exactly sandwich $K$. $w'(K)$ can therefore be interpreted as the average width of $K$ under the {Gaussian} measure, which is a geometric attribute of $K$ measuring its complexity. As was observed in \cite{Vershynin_2015},  Theorem \ref{T2} implies the following average relative recovery error estimate in $\ell_1$ minimization: 
\begin{Th}\label{T:l1}
\textcolor{black}{Let $\x^*$ be the solution to (\ref{l1}), and $\x_0$ be an $s$-sparse signal satisfying $\A\x_0 = \b$ and $\|\x_0\|_0 = s$.} Then, 
\begin{align}
\E\frac{\|\x^*-\x_0\|_2}{\|\x_0\|_2}\lesssim \sqrt{\frac{s\log n}{m}}\label{2},
\end{align}
where $a \lesssim b$ means that $a \leq C b$ for a universal constant $C$.
\end{Th} 

\begin{proof}
  Let $K_1=\|\x_0\|_1\cdot B^n_1$, where $B^n_1$ is the unit $\ell_1$ ball in $\R^n$. By definition, $\|\x^*\|_1\leq \|\x_0\|_1$ so that $\x^*\in K_1$. Therefore, $\x^*-\x_0\in (K_1-K_1)\cap\ker(\A)$. It follows immediately from (\ref{6}) with $K=K_1-K_1$ and $E=\ker(\A)$ that
\begin{align*}
  \E\|\x^*-\x_0\|_2&\leq\frac{1}{2}\E\text{\diam}(K\cap E)\leq \sqrt{\frac{2\pi}{m}}\cdot w'(K)\\
&\leq \sqrt{\frac{8\pi}{m}}\cdot w'(K_1) \lesssim \|\x_0\|_1\cdot\sqrt{\frac{\log n}{m}} \lesssim \|\x_0\|_2\cdot\sqrt{\frac{s\log n}{m}},
\end{align*}  
where the penultimate inequality uses $w'(B_1^n)=2w(B_1^n)$ and \eqref{l1-ball:gauss-width},  
and the last inequality follows from $\|\x_0\|_0\leq s$ and the Cauchy-Schwarz inequality. Dividing $\|\x_0\|_2$ on both sides finishes the proof. 
\end{proof}

\begin{Rem}
Taking $m\gtrsim s\log n$ in (\ref{2}) results in a bound on the right-hand side of (\ref{2}), which, up to logarithmic factors, achieves the desired statement that $m$ measurements can recover $s$-sparse signals. 
Note that the statement in \eqref{2} is only concerned with the \emph{average} relative error. One can go further to obtain a (pathwise) exact recovery result using {Gordon} Escape Theorem \cite{Vershynin_2015,Rudelson_2008}. However, the ideas from the proof of this result depend on the convex nature of the problem and are not extensible to non-convex cases, so we do not state it here.
\end{Rem}

An alternative approach to achieve an exact recovery condition for the $\ell_1$ minimization problem is to interpret the kernel of $\A$ as a random subspace under the Gaussian assumption of the measurements and is RIP-free \cite{Zhang_2013}. This method is similar to ideas described above in \cite{Vershynin_2015}. In fact, this RIP-free approach yields nearly all results that can be attained by RIP approaches. The analysis in \cite{Zhang_2013} is the inspiration for our approach, so we shall summarize its main results.


The first idea is to note that a sufficient condition for $\A$ to satisfy the $(s,1)$-NSP is given by
\begin{align}
\inf_{0\neq\h\in\ker(\A)}\frac{\|\h\|_1}{\|\h\|_2}>2\sqrt{s}.\label{zhang-l1-l2} 
\end{align} 
The condition (\ref{zhang-l1-l2}) is concerned with the ratio between the $\ell_1$ and $\ell_2$ norms in a random subspace of dimension $n-m$, which can be analyzed using the tools from high-dimensional geometry. Indeed, a classical result 
in geometric functional analysis states that if the measurement $\A$ has iid Gaussian entries, then 
\begin{align}
\inf_{0\neq\h\in\ker(\A)}\frac{\|\h\|_1}{\|\h\|_2}>\frac{c\sqrt{m}}{\sqrt{1+\log (n/s)}}\label{KGG}
\end{align}
holds with overwhelming probability, where $c$ is a dimension-free constant \cite{Gluskin_1984,kasin1977widths}. Relations (\ref{zhang-l1-l2}) and (\ref{KGG}) together give a bound on $m$ that is asymptotically equivalent to the classic result in \cite{Candes_2006}. 

A condition similar to (\ref{zhang-l1-l2}) is given in \cite{Zhang_2013} to guarantee that $\ell_1$ minimization is stable. A specialization of the result reads as follows:
\begin{Th}\label{zhang-stability}
Let $\tilde{\x}$ be the solution to the following minimization problem:
\begin{align*}
\min \|\x\|_1 \ \ \ \text{subject to}\ \|\A\x-\b\|_2\leq\e,
\end{align*}
where $\b \coloneqq \A\x_0+\mathbf{e}$ with $\|\mathbf{e}\|_2\leq\e$ and $\|\x_0\|_0\leq s$. Let $\u$ and $\w$ be the orthogonal projections of $\tilde{\x}-\x_0$ to $\ker(\A)$ and $\ker^\perp(\A)$, respectively. If 
\begin{align*}
s=\frac{v^2}{4}\frac{\|\u\|_1^2}{\|\u\|_2^2}
\end{align*}
for some $v\in (0, 1)$, {then for either} $p=1$ or $p=2$,
\begin{align*}
\|\tilde{\x}-\x_0\|_p\leq 2\gamma_p\left(1+\frac{1+v\sqrt{2-v^2}}{1-v^2}\right)\|\w\|_2, 
\end{align*}
where $\gamma_1=\sqrt{n}$ and $\gamma_2=1$. 
\end{Th}
It was also shown in \cite{Zhang_2013} that $\|\w\|_2$ can be further bounded by $\e\|\mathbf{R}^{-T}\|_2$, where $\mathbf{R}$ is the triangular matrix in the QR decomposition of $\A^T$. It is worth noting that Theorem \ref{zhang-stability} is not implied by the RIP results in \cite{Candes_2006}. In fact, the constants involved in Theorem \ref{zhang-stability} are more revealing since they are directly related to the sparsity level $s$ rather than the RIP parameters, which are not invariant under invertible transforms.

\section{A local optimality criteria}\label{local optimality}

In this section, we give a sufficient condition for an $s$-sparse signal $\x_0$ to be the local minimizer of \eqref{l1/l2} with $\b \coloneqq \A \x_0$. Compared to the global optimality condition obtained later in Section \ref{exact recovery}, the local optimality condition in this section aids in understanding the behavior of $\ell_1/\ell_2$ optimization near $\x_0$. This is important in practice since many non-convex algorithms only have local convergence guarantees. 
\textcolor{black}{Our local optimality result is signal-dependent but it offers asymptotically weaker and more interpretable conditions than those in \cite{Rahimi_2019}. }
Our characterization of local optimality depends on the (inverse) \emph{dynamic range} $\rho = \rho(\x)$ of a nontrivial vector $\x$, defined as
\begin{align}\label{eq:dynamic-range}
  \rho \coloneqq \frac{\min_{i\in\text{supp}(\x)}|x_i|}{\max_{i\in\text{supp}(\x)}|x_i|} = \frac{\min_{i\in\text{supp}(\x)}|x_i|}{\|\x\|_\infty}
\end{align}
Smaller values of $\rho$ indicate larger variation in the magnitude of the extremal nonzero entries in $\x_0$. It was observed in \cite{Wang_2020} that the performance of $\ell_1/\ell_2$ improves when the dynamic range increases. We will also need a quantity $\kappa$ that is ratio of norm ratios:
\begin{align}\label{kappa}
  \kappa = \kappa(\x_0) \coloneqq \frac{\|\x_0\|_1\|\x_0\|_\infty}{\|\x_0\|^2_2} = \frac{\frac{\|\x_0\|_1}{\|\x_0\|_2}}{\frac{\|\x_0\|_2}{\|\x_0\|_\infty}}.
\end{align}
The main result is the following:

\begin{Th}[Local optimality]\label{T:local}
Let $\x_0$ be a nonzero $s$-sparse vector ($s>1$), and 
$\A$ be a measurement matrix. Define
\begin{align}\label{eq:c-def}
  c = c(\A) \coloneqq \sup_{0\neq \h\in\ker(\A)}\frac{\|\h\|^2_2}{\|\h\|_1^2}.
\end{align}
Suppose that $\x_0, \A$ are such that
\begin{align}
  \rho (\kappa + 1) \leq\frac{1}{2c},\label{ratio}
\end{align}
where $\rho = \rho(\x_0)$ and $\kappa = \kappa(\x_0)$,
and that $\A$ satisfies the NSP with parameters $(s, \frac{1}{2\kappa+1})$, i.e.,  
\begin{align}\label{k-NSP}
  \|\h_T\|_1<\frac{1}{2\kappa+1}\|\h_{T^\complement}\|_1\ \ \ \textrm{for every } \h \in \ker(\A) \textrm{ and } T\subset[n]_s 
\end{align}
  Then $\x_0$ is the local minimizer of the constrained $\ell_1/\ell_2$ objective function with $\ell_1$ convergence radius $\delta = \rho\, \|\x_0\|_\infty$, 
  i.e., for any $\x \in \R^n$ satisfying $\A\x=\A\x_0$ and $\|\x-\x_0\|_1\leq\delta$, then $\|\x_0\|_1/\|\x_0\|_2<\|\x\|_1/\|\x\|_2$.    
\end{Th}
We prove this theorem later in this section, but first focus on some of its consequences. Theorem \ref{local optimality} is initially difficult to fully comprehend since the conditions (\ref{ratio}) and (\ref{k-NSP}) not only depend on the measurement matrix $\A$ but also on the sparse vector $\x_0$. However, a specialization is more transparent: A worst-case upper bound for $\kappa$ results in a local optimality condition that is uniformly true for all $s$-sparse vectors. 

\begin{Cor}\label{improved-local}
Assume $s > 6$. If $\A$ satisfies the NSP with parameters $(s, \frac{1}{\sqrt{s}+2})$, then $\x_0$ is a local minimizer of $\ell_1/\ell_2$ for all $\|\x_0\|_0\leq s$. 
\end{Cor}

\begin{proof}
Note that for any $s$-sparse $\x_0$, 
\begin{align}\label{eq:krho-uniform}
  \kappa &\leq \frac{(\sqrt{s}+1)}{2}, & \rho &\leq 1.
\end{align}
\textcolor{black}{Indeed, since $\kappa (\x_0)$ is invariant under both scaling and permutation, we may assume $\x_0 = (x_1, \cdots, x_s, 0, \cdots, 0)^T$ with $x_1\geq\cdots\geq x_s\geq 0$ and $\|\x_0\|_2=1$. 
In this case, $\kappa (\x_0) = x_1^2 + x_1(\sum_{i = 2}^s x_i)\leq x_1^2 + x_1\sqrt{(s-1)(1-x_1^2)}$, with equality achieved at $x_2 = \cdots = x_s$. 
The maximum of the right-hand side is $(\sqrt{s}+1)/2$ and is attained when $x_1 = 1/2 + 1/(2\sqrt{s})$. }

Since $\A$ satisfies the $(s, \frac{1}{\sqrt{s}+2})$-NSP, then 
\begin{align}\label{002}
  \|\h_T\|_1<\frac{1}{\sqrt{s}+2}\|\h_{T^\complement}\|_1\ \ \ \forall (\h, T) \in \ker(\A) \times [n]_s, 
\end{align}
which combines with \eqref{eq:krho-uniform} to establish \eqref{k-NSP}.
Now we note that if
\begin{align}
c\leq\frac{1}{\sqrt{s}+3}. \label{001}
\end{align}
is achieved for all $s$-sparse vectors, then this along with \eqref{eq:krho-uniform} implies \eqref{ratio}. We claim that \eqref{002} implies \eqref{001} for $s>6$: Let $\h \in \ker(\A)$ and note that \eqref{002} implies that
\begin{align*}
  \|\h\|^2_1 &\geq (\sqrt{s}+3)^2\|\h_T\|_1^2, & \h &\in \ker(\A).
\end{align*}
For $r \in \N$, let $T_r$ be the support of the $((r-1)s+1)$-th to the $rs$-th components of $\h$ arranged in decreasing magnitude. This partition ensures that
\begin{align}\label{eq:h-partition}
  \|\h_{T_r}\|_\infty \leq \min_{i \in T_{r-1}} |h_i| \leq \frac{1}{s} \sum_{i \in T_{r-1}} |h_i| = \frac{1}{s} \| \h_{T_{r-1}}\|_1.
\end{align}
Then applying a block-type argument, for $s>6$,  
\begin{align*}
\|\h\|^2_2= \|\h_{T_1}\|^2_2+\sum_{r\geq 2}\|\h_{T_r}\|^2_2&\leq \frac{1}{(\sqrt{s}+3)^2}\|\h\|_1^2+\sum_{r\geq 2}\|\h_{T_r}\|^2_2\\
&\stackrel{\eqref{eq:h-partition}}{\leq} \frac{1}{(\sqrt{s}+3)^2}\|\h\|_1^2+\sum_{r\geq 2} \sum_{i=1}^s \left( \frac{\|\h_{T_{r-1}}\|_1}{s}\right)^2\\
&\leq \frac{1}{(\sqrt{s}+3)^2}\|\h\|_1^2+ \frac{1}{s} \sum_{r\geq 1} \|\h_{T_{r}}\|^2_1\\
&= \frac{1}{(\sqrt{s}+3)^2}\|\h\|_1^2 + \frac{1}{s} \|\h\|_1^2 \leq \frac{1}{\sqrt{s}+3}\|\h\|_1^2.
\end{align*}
We have thus established both \eqref{ratio} and \eqref{k-NSP}, so that the conclusion of Theorem \ref{T:local} holds.
\end{proof}

A similar technique does not, unfortunately, provide a uniform bound on the local convergence radius due to technical issues.
Without asking for uniformity of the local convergence radius, Corollary \ref{improved-local} gives an asymptotic weaker condition for local optimality of sparse vectors compared to the result in \cite{Rahimi_2019}, which requires a stronger NSP with parameters $(s, \frac{1}{s+1})$ for $s\geq 1$. In fact, in many situations of interest, $\kappa$ is of mild order (which is specified in the following proposition), suggesting that (\ref{k-NSP}) is not as restrictive as in Corollary \ref{improved-local}. 

\begin{Prop}
\textcolor{black}{Suppose that the $s$ nonzero components of $\x_0$ are iid with the same distribution as a scalar centered sub-gaussian random variable $X$.  
Then, for sufficiently large $s$, the following event holds with probability at least $1-1/s^2$:
\begin{align*}
  \kappa(\x_0) \leq C\sqrt{\log s},
\end{align*} 
where $C$ is a constant depending only on the sub-gaussian norm of $X/\sqrt{\E X^2}$.}
\end{Prop}
\begin{proof}
\textcolor{black}{Since $\kappa$ is scale-invariant, we may assume that $\E X^2=1$, i.e., $\E|X|\leq 1$ and $\text{Var} X\leq 1$. 
Denote the sub-gaussian norm (see Definition \ref{def:subg}) of $X$ as $\sigma_X$.
Applying Hoeffding's inequality to $\|\x_0\|_1$, Bernstein's inequality to $\|\x_0\|_2^2$ and the maximal inequality \cite[Theorem 1.14]{rigollet2015high} to $\|\x_0\|_\infty$ yields that for sufficiently large $s$, 
\begin{align*}
\P\left(\|\x_0\|_1>2 s\right)&\leq \P\left(\|\x_0\|_1-\E\|\x_0\|_1>s\E|X|\right)\leq 2e^{-cs}\leq 1/3s^2\\
\P\left(\|\x_0\|^2_2<s/2\right)&\leq \P\left(\|\x_0\|^2_2 - \E \|\x_0\|^2_2\leq -s\E X^2/2\right)\leq 2e^{-cs}\leq 1/3s^2\\
\P\left(\|\x_0\|_\infty> 4\sigma_X\sqrt{\log s}\right)& \leq  \P\left(\|\x_0 -\E\x_0\|_\infty > 3\sigma_X\sqrt{\log s}\right)\leq 1/3s^2, 
\end{align*}
where $c$ is a constant depending only on $\sigma_X$.
It follows from a union bound that 
\begin{align*}
\P\left(\kappa (\x_0)\leq 16\sigma_X\sqrt{\log s}\right)\geq 1-1/s^2.
\end{align*}
The proof is complete.}
\end{proof}

We now provide the proof of Theorem \ref{T:local}:

\begin{proof}[Proof of Theorem \ref{T:local}]
The main idea of the proof is that under condition \eqref{k-NSP}, a large proportion of the perturbation of $\x_0$ by $\h\in\ker(\A)$ will be well-spread outside the support of $\x_0$, which necessarily increases the value of the objective $\ell_1/\ell_2$. 

For any $\x, \y \neq 0$, define the ordering $\x\succcurlyeq (\succ) \y$ as $\frac{\|\x\|_1}{\|\x\|_2}\geq (>) \frac{\|\y\|_1}{\|\y\|_2}$. For simplicity and without loss of generality, we assume the nonzero entries of $\x_0$ are arranged in order of decreasing magnitude in the first $s$ components, i.e., that $\x_0=(x_1, \cdots, x_s, 0, \cdots, 0)^T\in\R^n$ with $|x_1|\geq \cdots\geq |x_s|>0$. Define $\delta \coloneqq \rho \|\x_0\|_\infty = |x_s| > 0$.

We claim that for any $\h=(h_1, \cdots, h_n)^T\in\ker(\A)$ with $\|\h\|_1\leq\delta$, the perturbed vector $\x$ satisfies
\begin{align*}
\x \coloneqq \x_0+\h=(x_1+h_1, \cdots, x_s+h_s, h_{s+1}, \cdots, h_n)^T\succ \x_0, 
\end{align*} 
which would prove the desired result. To show the above relation, we will construct another vector $\x'$ from $\x$ and establish the ordering,
\begin{align}\label{eq:porder}
  \x \succcurlyeq \x' \succcurlyeq \x_0.
\end{align}
To begin, introduce $\beta =\sum_{i=1}^s \text{sgn}(x_i)h_i$ and $\gamma =\sum_{i=1}^s|h_i|$ , and augment entries $1$ and $s$ in $\x$ to obtain 
\begin{align*}
\x' \coloneqq \left(x_1+\text{sgn}(x_1)\frac{\gamma+\beta}{2}, x_2, \cdots, x_{s-1}, x_s-\text{sgn}(x_s)\frac{\gamma-\beta}{2}, h_{s+1}, \cdots, h_n\right)^T.
\end{align*}
Note that since $|\beta| \leq \gamma$, then 
  \begin{align}\nonumber
    \|\x'\|_1 &= \left| x_1 + \sgn(x_1) \frac{\gamma + \beta}{2}\right| + \left| x_s - \sgn(x_s) \frac{\gamma - \beta}{2}\right| + \sum_{j=2}^{s-1} |x_j| + \sum_{j=s+1}^n |h_j| \\\nonumber 
              &= \beta + \sum_{j=1}^s |x_j| + \sum_{j=s+1}^n |h_j| = \sum_{j=1}^s |x_j| + \sgn(x_j) h_j + \sum_{j=s+1}^n |h_j| \\\label{eq:xpx1}
              &= \sum_{j=1}^s \left| x_j + h_i \right| + \sum_{j=s+1}^n | h_j| = \| \x\|_1,
  \end{align}
where the penultimate equality uses the fact that the assumption $\|\h\|_1 \leq \delta$ implies $|h_i| \leq |x_i|$ for $i = 1, \ldots, s$.
To show that $\|\x'\|_2\geq \|\x\|_2$, we express the difference of their squares as
\begin{subequations}\label{eq:xpx2}
\begin{equation}
  \|\x'\|_2^2-\|\x\|_2^2 = \underbrace{(\gamma+\beta)|x_1|-(\gamma-\beta)|x_s| -2\sum_{i=1}^s\text{sgn}(x_ih_i)|x_i\|h_i|}_{(A)} + \underbrace{\frac{1}{2}(\gamma^2+\beta^2)-\sum_{i=1}^s h_i^2}_{(B)}
\end{equation}
Term (A) satisfies
\begin{equation}
  (A) = 2\sum_{\text{sgn}(x_ih_i)=1}(|x_1|-|x_i|)|h_i|+2\sum_{\text{sgn}(x_ih_i)=-1}(|x_i|-|x_s|)|h_i| \geq 0,
\end{equation}
and term (B)
\begin{align}\nonumber
  (B) &= \frac{1}{2}\left(2\sum_{i=1}^s h_i^2+2\sum_{1\leq i<j\leq s}\left(|h_i\|h_j|-\text{sgn}(x_ix_j)h_ih_j\right)\right)-\sum_{i=1}^sh_i^2 \\
      &= \sum_{1\leq i<j\leq s}\left(|h_i\|h_j|-\text{sgn}(x_ix_j)h_ih_j\right) \geq 0.
\end{align}
\end{subequations}
Relations \eqref{eq:xpx1} and \eqref{eq:xpx2} establish the upper ordering in \eqref{eq:porder}.
To show the lower ordering, we first note that Taylor's Theorem applied to the function $y \mapsto \sqrt{y}$ along with that function's concavity implies that for any $\beta > 0$:
\begin{align}\label{eq:taylor-temp}
  \sqrt{\|\x_0\|_2^2 + \beta} \leq \|\x_0\|_2 + \frac{1}{2 \|\x_0\|_2} \beta.
\end{align}
Now we directly compare the $\ell_1/\ell_2$ norms of $\x_0$ and $\x'$:
\begin{align*}
\frac{\|\x'\|_1}{\|\x'\|_2}&=\frac{\|\x_0\|_1+\beta+\sum_{i=s+1}^n|h_i|}{\sqrt{\|\x_0\|_2^2+(\gamma+\beta)|x_1|-(\gamma-\beta)|x_s|+\frac{1}{2}(\gamma^2+\beta^2)+\sum_{i=s+1}^nh_i^2}}\\
                           &\stackrel{|\beta|\leq \gamma}{\geq}\frac{\|\x_0\|_1-\gamma+\sum_{i=s+1}^n|h_i|}{\sqrt{\|\x_0\|_2^2+2\gamma |x_1|+\gamma^2+\|\h\|_2^2}}\\
                           &\stackrel{\eqref{eq:taylor-temp}}{\geq}\frac{\|\x_0\|_1-\gamma+\sum_{i=s+1}^n|h_i|}{\|\x_0\|_2+\frac{1}{2\|\x_0\|_2}(2\gamma |x_1|+\gamma^2+\|\h\|_2^2)}\\
                           &\leftstackrel{\eqref{k-NSP},\eqref{eq:dynamic-range},\eqref{eq:c-def}}{>}\frac{\|\x_0\|_1+\frac{2\kappa}{2\kappa+2}\|\h\|_1}{\|\x_0\|_2+\frac{1}{2\|\x_0\|_2}(\frac{2}{2\kappa+2} \|\x_0\|_\infty\|\h\|_1+\frac{\rho}{(2\kappa+2)^2}\|\x\|_\infty\|\h\|_1+c\rho\|\x\|_\infty\|\h\|_1)}\\
&\geq\min\left(\frac{\|\x_0\|_1}{\|\x_0\|_2}, \frac{\frac{4\kappa}{2\kappa+2}}{\frac{2}{2\kappa+2}+\frac{\rho}{(2\kappa+2)^2}+c\rho}\frac{\|\x_0\|_2}{\|\x_0\|_\infty}\right)\\
&\stackrel{\eqref{kappa}}{\geq}\min\left(1, \frac{\frac{4}{2\kappa+2}}{\frac{2}{2\kappa+2}+\frac{\rho}{(2\kappa+2)^2}+c\rho}\right)\frac{\|\x_0\|_1}{\|\x_0\|_2}\\
&\stackrel{\eqref{ratio}}{=}\frac{\|\x_0\|_1}{\|\x_0\|_2}.  
\end{align*}

We have thus established the lower relation in \eqref{eq:porder} and the proof is complete.   
\end{proof}

\begin{Rem}
From Theorem \ref{T:local}, we notice that  (\ref{ratio}) is automatically satisfied if $\rho<1/(\sqrt{s}+3)$. This suggests that the local optimality criteria is more likely to hold for vectors with a larger dynamic range, which is a possible explanation for why large dynamic range aids in the numerical performance of $\ell_1/\ell_2$ algorithms \cite{Wang_2020}. We will also numerically investigate this in Section \ref{numerical experiments}.   
\end{Rem}
\textcolor{black}{
\begin{Rem}
It is possible to derive a sufficient and necessary condition for an $s$-sparse solution $\x_0$ to be a local minimum by taking directional derivatives, which is the approach taken in \cite{Rahimi_2019}. 
Without loss of generality we assume that $\|\x_0\|_2 = 1$. 
For $\h\in\ker(\A)$, consider the function $L_\h(t) = \|\x_0+t\h\|_1/\|\x_0+t\h\|_2, t\geq 0$. 
$\x_0$ is a local minimizer of the $\ell_1/\ell_2$ minimization subject to $\A\x=\A\x_0$ if and only if $L'_\h(0)\geq 0$ for all $\h\in\ker(\A)$, and this is equivalent to the following condition:
\begin{align*}
&\|\h_{S^\complement}\|_1\geq\langle (\|\x_0\|_1(\x_0)_S-\sgn((\x_0)_S), \h_S\rangle& \forall 0\neq \h\in\ker(\A),
\end{align*}
where $S$ is the support of $\x_0$. 
An unconditional upper bound on the right-hand side (independent of $\x_0$) is $(\sqrt{s}+1)\|\h_S\|_1$, from which one can deduce that the NSP with parameters $(s, \frac{1}{\sqrt{s}+1})$ is sufficient to guarantee the uniform local optimality of $s$-sparse signals. This implies the uniform local optimality result in Corollary \ref{improved-local}. 
However, the NSP condition is not tight since the equality cannot be attained, i.e., $\langle\sgn((\x_0)_S), \h_S\rangle = \|\h_S\|_1$ implies $\langle \|\x_0\|_1(\x_0)_S, \h_S\rangle< 0$.
As an alternative, Theorem \ref{T:local} makes an attempt to understand the local optimality of a given sparse signal based on signal-dependent structure. 
\end{Rem}
}

\section{A sufficient condition for exact recovery}\label{exact recovery}

In this section, we propose a sufficient condition that guarantees the uniform exact recovery for sparse vectors using $\ell_1/\ell_2$. As we will see, the condition to be obtained will hold with overwhelming probability for a large class of sub-gaussian random matrices. Since our approach for deriving this recoverability condition applies to other situations as well, we consider the following problem which is slightly more general than (\ref{l1/l2}): For $0<q\leq 1$ fixed, consider the optimization
\begin{align}
\min \frac{\|\x\|_q}{\|\x\|_2} \ \ \ \text{subject to}\ \A\x=\b.\label{q-opt}
\end{align} 
We will refer to \eqref{q-opt} as the $\ell_q/\ell_2$ minimization problem. The problem \eqref{q-opt} is equivalent to \eqref{l1/l2} when $q=1$. Clearly, (\ref{q-opt}) recovers all $s$-sparse vectors if and only if for any nonzero $\h\in \ker(\A)$ and $\x_0$ with $\|\x_0\|_0\leq s$,
\begin{align}
\frac{\|\x_0\|^q_q}{\|\x_0\|^q_2}<\frac{\|\x_0+\h\|^q_q}{\|\x_0+\h\|^q_2}.\label{9}
\end{align}   
We will now directly work with (\ref{9}) to find a sufficient condition establishing this property. 
\textcolor{black}{As we will see, the condition to be obtained from our analysis is strictly sufficient and \emph{stronger} than the NSP assumption used in the-state-of-art $\ell_q$ recovery (in particular with $q=1$). This is mainly due to the technical difficulty in analyzing the ratio form in a uniform way.} The main exact recovery result is as follows.

\begin{Th}[Uniform recoverability]\label{T5}
  If, for some $s \in \N$, the matrix $\A$ satisfies 
\begin{align}
\inf_{\h\in\ker(\A)\backslash\{\mathbf{0}\}}\frac{\|\h\|_q}{\|\h\|_2}>3^{1/q}s^{1/q-1/2}.\label{exact-rc}
\end{align}
then \eqref{9} holds, establishing exact recovery for the optimization \eqref{q-opt}.
\end{Th}

Condition (\ref{exact-rc}) should be compared to (\ref{zhang-l1-l2}) in Section \ref{sec:cs}. When $q=1$, the right-hand side of (\ref{exact-rc}) is $3\sqrt{s}$, which is slightly larger (worse) than $2\sqrt{s}$ in (\ref{zhang-l1-l2}). When $\A$ is Gaussian, (\ref{KGG}) ensures that (\ref{exact-rc}) is satisfied with overwhelming probability provided that $m$ behaves linearly in $s$ (up to a logarithmic factor). The next theorem generalizes (\ref{KGG}) to the class of isotropic sub-gaussian matrices, but with a slightly worse logarithmic factor. For convenience, we first recall the definition of sub-gaussian and isotropic vectors.

\begin{Def}[Sub-gaussian random vectors]\label{def:subg}
  A random vector $X\in\R^n$ is said to be sub-gaussian if its one-dimensional marginal $\a^TX$ is sub-gaussian for all $\a\in\R^n$. In other words, $X$ is sub-gaussian if $\|\a^TX\|_{\Psi_2} \coloneqq \inf\{t>0, \E[e^{|\a^TX|^2/t^2}]\leq 2\}<\infty$ for all $\a\in\R^n$, where the sub-gaussian norm $\|\cdot\|_{\Psi_2}$ of $X$ is defined as 
\begin{align*}
\|X\|_{\Psi_2}=\sup_{\|\a\|_2=1}\|\a^TX\|_{\Psi_2}. 
\end{align*} 
\end{Def} 
If $X$ is standard normal, all of its marginals are standard normal in 1D, therefore $\|X\|_{\Psi_2} = \|\mathcal{N}(0,1)\|_{\Psi_2}$.

\begin{Def}[Isotropic random vectors]
A random vector $X\in\R^n$ is said to be isotropic if $\E[XX^T]=\mathbf{I}_n$, where $\mathbf{I}_n$ is the identity matrix. 
\end{Def}

\begin{Th}[Uniform recoverability for sub-gaussian matrices]\label{T4}
Let $\A\in\R^{m\times n}$ be a random matrix whose rows are independent, isotropic and sub-gaussian random vectors in $\R^n$. Suppose that $s$ satisfies, 
\begin{align}\label{eq:ms-sub-gaussian}
  \frac{m}{s} > D F^4 u \log n,
\end{align}
where $F$ is the maximum sub-gaussian norm of rows of $\A$, $u \geq 1$, and $D$ is an absolute constant. Then, for $q=1$, \eqref{exact-rc} holds with probability at least $1-2e^{-u}$.
\end{Th}

The class of sub-gaussian random matrices considered in Theorem \ref{T4} is general enough for practical purposes: Gaussian, symmetric Bernoulli, and many other random matrices whose entries are sampled from bounded distributions with appropriate centralization and normalization fall into this realm.

We now provide the proof of Theorem \ref{T5}. 

\begin{proof}[Proof of Theorem \ref{T5}]
To find a sufficient condition only, we may require the left-hand side of (\ref{9}) to be smaller than a quantity which is unconditionally smaller than the right-hand side of (\ref{9}). First recall the $q$-triangle inequality, which states that for $0<q\leq 1$, 
\begin{align*}
\|\x_1+\x_2\|_q^q\leq \|\x_1\|_q^q+\|\x_2\|_q^q, \ \ \ \forall \x_1, \x_2\in\R^n.
\end{align*}
Fix an $\x_0$ such that $\|\x_0\|_0 \leq s$, and let $S \subset [n]$ denote the support of $\x_0$. The right-hand side of (\ref{9}) satisfies   
\begin{align}
\frac{\|\x_0+\h\|^q_q}{\|\x_0+\h\|^q_2}&\geq\frac{\|(\x_0+\h)_S\|^q_q+\|\h_{S^\complement}\|^q_q}{(\|\x_0\|_2+\|\h\|_2)^q}\nonumber\\
&\geq\frac{\|\x_0\|_q^q+\|\h_{S^\complement}\|^q_q-\|\h_S\|_q^q}{\|\x_0\|^q_2+\|\h\|^q_2},\label{16}
\end{align}
where for the first inequality we used $\|\x_0+\h\|_2\leq \|\x_0\|_2+\|\h\|_2$, and for the second inequality, we used the $q$-triangle inequality for the numerator, $\|(\x_0+\h)_S\|_q^q\geq \|\x_0\|_q^q-\|\h_S\|_q^q$, and concavity of $y \mapsto y^q$ for the denominator, $(\|\x_0\|_2+\|\h\|_2)^q\leq \|\x_0\|_2^q+\|\h\|_2^q$. Continuing, we have
\begin{align*}
  \frac{\|\x_0\|^q_q+\|\h_{S^\complement}\|^q_q-\|\h_S\|^q_q}{\|\x_0\|^q_2+\|\h\|^q_2}  &\geq \min\left\{\frac{\|\x_0\|_q^q}{\|\x_0\|_2^q}, \frac{\|\h_{S^\complement}\|^q_q-\|\h_S\|^q_q}{\|\h\|^q_2}\right\}\\
  & = \min\left\{\frac{\|\x_0\|_q^q}{\|\x_0\|_2^q}, \frac{\|\h\|^q_q-2\|\h_S\|^q_q}{\|\h\|^q_2}\right\}
\end{align*} 
for all nonzero $\h\in\ker(\A)$. Thus, we have established that 
\begin{subequations}\label{eq:thm-T5-temp}
\begin{align}
  \frac{\|\x_0+\h\|^q_q}{\|\x_0+\h\|^q_2} \geq \min\left\{\frac{\|\x_0\|_q^q}{\|\x_0\|_2^q}, \frac{\|\h\|^q_q-2\|\h_S\|^q_q}{\|\h\|^q_2}\right\}, 
\end{align}
where equality holds if and only if $\frac{\|\x_0\|_q^q}{\|\x_0\|_2^q} = \frac{\|\h\|^q_q-2\|\h_S\|^q_q}{\|\h\|^q_2}$. Now by the generalized H\"older's inequality, we have
\begin{align}\label{eq:gen-holder}
  s^{1 - q/2} \geq \frac{\|\x_0\|^q_q}{\|\x_0\|^q_2}
\end{align}
\end{subequations}
which is a sharp estimate since we consider a general $s$-sparse $\x_0$. Therefore the conditions \eqref{eq:thm-T5-temp} in tandem with
\begin{align}
\frac{\|\h\|_q^q-2\|\h_S\|_q^q}{\|\h\|_2^q}>s^{1-q/2}.\label{11}
\end{align}   
are sufficient to conclude Theorem \ref{T5}. In order to prove \eqref{11}, we note that under assumption \eqref{exact-rc} then
\begin{align*}
  \frac{\|\h\|_q^q}{\|\h\|_2^q} &\stackrel{\eqref{exact-rc}}{>} 3 s^{1 - q/2} 
  \stackrel{\eqref{eq:gen-holder}}{\geq} s^{1 - q/2} + 2 \frac{\|\h_S\|_q^q}{\|\h_S\|_2^q} \\
                &\geq  s^{1 - q/2} + 2 \frac{\|\h_S\|_q^q}{\|\h\|_2^q},
\end{align*}
where the second inequality uses the generalized H\"older inequality \eqref{eq:gen-holder} for the $s$-sparse vector $\h_S$. 
This inequality is equivalent to \eqref{11}, proving Theorem \ref{T5}.
\end{proof}

\begin{Rem}
  In the derivation above, the condition \eqref{11} is very different from \eqref{9}. Since \eqref{11} is scale-invariant, it holds for all $\h\in\ker(\A)$ if and only if it holds for $\h\in\ker(\A)\cap K$, where $K$ is any $\mathbf{0}$-starshaped set in $\R^n$. E.g., $K$ can be the unit ball in $\ell_q$. However, for \eqref{9}, it is generally not true that $\|\x_0\|_q/\|\x_0\|_2<\|\x_0+\h\|_q/\|\x_0+\h\|_2$ implies the same inequality for $\h$ replaced by $t \h$, $t \in \R\backslash\{0\}$. This suggests that \eqref{11} is strictly stronger than \eqref{9}.
\end{Rem}

Next we give the proof for Theorem \ref{T4}, based on a matrix deviation inequality inspired by \cite{Vershynin_2018}.   

\begin{proof}
Under the assumption of Theorem \ref{T4}, it follows from \cite[Exercise 9.13]{Vershynin_2018} that for any bounded $T\subset\R^n$ and $u>0$, with probability at least $1-2e^{-u^2}$,  
\begin{align}\label{eq:versh-estimate}
\sup_{\x\in T}\bigg\||\A\x\|_2-\sqrt{m}\|\x\|_2\bigg|\leq cF^2(\gamma(T)+u\cdot\text{rad}(T)),
\end{align}
where 
\begin{align*}
  \gamma(T) \coloneqq \E\sup_{\x\in K}|\langle\g, \x\rangle|
\end{align*}
is a variant of the Gaussian width $w(T)$ from Definition \ref{Gaussian complexity},
$\text{rad}(T)$ is the radius of $T$, $c$ is some absolute constant, and $F=\max_i\|\A_i\|_{\psi_2}$ with $\A_i$ the $i$-th row of $\A$. 

Take $T=\ker(\A)\cap B_1^n$. In this case, $T$ is symmetric with respect to the origin, so we have 
\begin{align*}
  \gamma(T)& = w(T)\leq w(B_1^n) \stackrel{\eqref{l1-ball:gauss-width}}{\leq} \sqrt{8\log n}< 2 \sqrt{\pi \log n},\\
  \text{rad}(T)&=\frac{1}{2}\diam(T)\leq\frac{\sqrt{2\pi}}{2}w(T) \stackrel{\eqref{l1-ball:gauss-width}}{\leq} 2\sqrt{\pi \log n}.
\end{align*}
Using these in \eqref{eq:versh-estimate}, we have the following inequality with probability at least $1-2e^{-u^2}$:
\begin{align}\label{eq:versh-estimate-2}
  \sup_{\x\in\ker(\A)\cap B_1^n}\|\x\|_2\leq 2 cF^2\sqrt{\pi} \left(1+u\right)\sqrt{\frac{\log n}{m}} \leq 4 cF^2 \sqrt{\pi} u \sqrt{\frac{\log n}{m}},
\end{align}
where the last inequality uses $u \geq 1$. Thus, with probability at least $1-2e^{-u^2}$,
\begin{align*}
  \inf_{\h\in\ker(\A)\backslash\{\mathbf{0}\}}\frac{\|\h\|_1}{\|\h\|_2} \stackrel{\eqref{eq:versh-estimate-2}}{\geq} \frac{1}{4 cF^2\sqrt{\pi} u}\sqrt{\frac{m}{\log n}} 
  \stackrel{\eqref{eq:ms-sub-gaussian}}{>} 3 \sqrt{s},
\end{align*}
where we have taken $D = 144 \pi c^2$ in our use of \eqref{eq:ms-sub-gaussian}. Renaming $u^2$ by $u$ leads to the final inequality \eqref{exact-rc}.
\end{proof}

\begin{Rem}
Theorem \ref{T4} only shows that \eqref{exact-rc} is satisfied with high probability for isotropic
subgaussian matrices when $q = 1$. 
It is interesting to know if the same holds true for all $q<1$ and if the corresponding constant will get better.  
The answer to this question is not completely known to us, but there is some evidence that it might be true. 
Indeed, for fixed $q\leq 1$, if $\A$ satisfies
\begin{align}
\inf_{\x\in\ker(\A)\backslash\{\mathbf{0}\}}\frac{\|\x\|_q}{\|\x\|_2}\geq c_q^{1/q}\left(\frac{m}{\log(n/m)+1}\right)^{1/q-1/2}\label{gf}
\end{align}
for some $c_q>0$, then 
\begin{align*}
m\geq \left(\frac{c_q}{3}\right)^{1-q/2}s\log n
\end{align*}
is sufficient for \eqref{exact-rc}. The right-hand side of \eqref{gf} is closely related to the Gelfand's widths of $\ell_p$ balls, see \cite{foucart2010gelfand}. Assuming that $\A$ is a Gaussian random matrix and $n\geq m^2$, a result in \cite{Donoho_2006} states that there exists $c_q>0$ such that \eqref{gf} is satisfied with probability approaching $1$ as $n\to \infty$. However, it is not clear whether the best attainable $c_q$ in \eqref{gf} ensures  $(c_q/3)^{1-q/2}$ as a decreasing function in $q$.  
\end{Rem}

\section{Robustness analysis}\label{stability}

In this section, we discuss the robustness of $\ell_1/\ell_2$ minimization when noise is present. As in other compressed sensing results, we assume that $\b$ is contaminated by some noise $\mathbf{e}$: $\b=\A\x_0+\mathbf{e}\in\R^m$, where $\x_0$ is a sparse vector. If the size (say the $\ell_2$ norm) of $\mathbf{e}$ is bounded by an \textit{a priori} known quantity $\e$, then the $\ell_1/\ell_2$ denoising problem can be stated as 
\begin{align}
\min\frac{\|\x\|_1}{\|\x\|_2} \ \ \ \text{subject to}\ \|\A\x-\b\|_2\leq\e.\label{777}
\end{align}

Let $\x^*$ be a minimizer of (\ref{777}). Then $\|\A\x^*-\A\x_0\|_2\leq 2\e$, and necessarily, 
\begin{align}
\frac{\|\x^*\|_1}{\|\x^*\|_2}\leq\frac{\|\x_0\|_1}{\|\x_0\|_2}\leq\sqrt{s}. \label{key}
\end{align}
This inequality will play an important role in our following discussion. Since $\ell_1/\ell_2$ is scale-invariant, it would be difficult to distinguish $\x^*$ from $\x_0$ when $\x^*$ is nearly parallel to $\x_0$ with similar magnitude. This behavior is quantified in our main robustness result:
\begin{Th}[Robustness]\label{T6}
Let $\x_0\in\R^n$ be an $s$-sparse vector and $\x^*$ be a minimizer of (\ref{777}). Let $\x^*-\x_0=\u+\w$ with $\langle \u, \w\rangle=0$ and assume that 
\begin{align}\label{eq:beta-def}
  \beta \coloneqq 4 \sqrt{2 s} \frac{\| \u\|_2}{\|\u\|_1} < 1.
\end{align}
Let $\alpha \in (\beta, 1)$ be any number. The following holds:
\begin{itemize}
  \item If both of the conditions 
\begin{subequations}\label{eq:cond}
  \begin{align}\label{eq:cond-1}
    \langle \x_0, \x^*\rangle &\geq (1-\alpha^2/2)\|\x_0\|_2\|\x^*\|_2, \\\label{eq:cond-2}
    \|\x_0\|_2 &\leq \|\x^*\|_2 \leq (1+\alpha)\|\x_0\|_2,
  \end{align}
\end{subequations}
hold, then 
\begin{align}
\|\x^*-\x_0\|_2\leq 2\sqrt{\alpha}\|\x_0\|_2. \label{case-relative}
\end{align}  
\item If at least one of \eqref{eq:cond} is violated, then 
\begin{align}
    \|\x^*-\x_0\|_p\leq \frac{2\alpha - \beta}{\alpha - \beta}\|\w\|_p,\label{stab-2-est}
\end{align}
for either $p=1$ or $p=2$.
\end{itemize}

\end{Th}
\textcolor{black}{We provide some intuitive interpretation of the results in Theorem \ref{T6}}.
The conditions \eqref{eq:cond} codify the regime when suboptimal behavior of $\ell_1/\ell_2$ optimization is expected due to a confounding geometric positions of $\x^*$ and $\x_0$.
\textcolor{black}{If $\eqref{eq:cond}$ is violated, then $\x^*$ either points to a different direction than $\x_0$ or has a different magnitude. 
In either case, it can be shown that $\|\x^*-\x_0\|_1/\|\x^*-\x_0\|_2$ is reasonably small. 
This combined with the orthogonal decomposition $\x^*-\x_0=\u+\w$ and the fact that $\u$ has large $\ell_1/\ell_2$ ratio implies that $\|\u\|\ll \|\w\|$ under some appropriate norm. This gives an informal justification for \eqref{stab-2-est}.
}
Note that \eqref{stab-2-est} provides an upper bound on $\|\x^* - \x_0\|_2$ relative to $\|\w\|_1$ regardless of the value of $p$ since $\|\cdot\|_2 \leq \|\cdot\|_1$.
We demonstrate the utility of Theorem \ref{T6} by showing how it can be used to show the robustness of \eqref{777}.
\textcolor{black}{\begin{Cor}\label{cor:C6}
Let $\A \in \R^{m \times n}$ be an isotropic sub-gaussian random matrix, 
and let $\x^*$ be a solution to \eqref{777}. 
  Let $\alpha \in (0,1)$ be such that
  for some $u> \log 5$,
  \begin{align}\label{eq:stab-sample}
   \frac{n}{4}+1\geq m \geq K \frac{u}{\alpha^2}\, s \, \log n
  \end{align}
  where $K$ is a constant depending only the sub-gaussian distribution of $\A$. Then with probability exceeding $1 - 5 e^{-c u}$, for all $s$-sparse vectors $\x_0 \in \R^n$, at least one of the following is true: (i) \eqref{case-relative} is true, or (ii) 
    \begin{align}\label{eq:stab-cor-l2}
    \|\x^*-\x_0\|_p &\leq C_p\e, 
  \end{align}
  is true for either $p=1$ or $p=2$, where $C_1 =  \sqrt{n}C_2$, and $c, C_2$ are constants depending only on the sub-gaussian distribution of $\A$.
\end{Cor}}
\begin{proof}
  Decompose $\x^*-\x_0=\u+\w$ with $\u \in \ker(\A)$ and $\w \in \ker(\A)^\perp$. By assumption \eqref{eq:stab-sample}, we can apply Theorem \ref{T4} to conclude that with probability at least $1 - 2 e^{-u}$, we have
  \begin{align}
    \frac{\|\u\|_1}{\|\u\|_2}\geq\frac{8\sqrt{2 s}}{\alpha} \hskip 5pt \Longrightarrow \hskip 5pt \frac{\beta}{\alpha} \leq \frac{1}{2},\label{ne}
  \end{align}
  with $\beta$ as in \eqref{eq:beta-def}. Note in particular that this implies $\beta < 1$. We can now apply Theorem \ref{T6}, so that either \eqref{case-relative} holds (as desired) or \eqref{stab-2-est} holds. We now investigate \eqref{stab-2-est}: 
  \begin{align}\label{eq:C6-temp}
    \|\x-\x_0\|_p\leq \frac{2 \alpha - \beta}{\alpha - \beta} \|\w\|_p \stackrel{\eqref{ne}}{\leq} 3 \|\w\|_p,
  \end{align}
  for either $p = 1$ or $p = 2$. In order to compute bounds for $\|\w\|_p$, we appeal to a well-known result on the lower bound of the singular value of sub-gaussian random matrices: Theorem 1.1 in \cite{Rudelson_2009} shows that the smallest nonzero singular value $\s_{\mathrm{min}}(\A)$ of $\A$ is bounded below by some positive constant with high probability. More precisely, 
  \begin{align}
    \P\left(\s_n(\A)\geq 0.5C\left(1-\sqrt{\frac{m-1}{n}}\right)\right)\geq 1-e^{-cn}-2^{-n-m+1} \geq 1 - 3 e^{-c_1 n},\label{stab-subg}
  \end{align}
  where $c, c_1, C>0$ are absolute constants only depending on the sub-gaussian distribution of $\A$. Since $u < n$ (otherwise \eqref{eq:stab-sample} implies more measurements than unknowns), then the events \eqref{stab-subg} and \eqref{ne} occur simultaneously with probability at least $1 - 5 e^{-c u}$ for some constant $c$. Under this simultaneous event, with $\A^\dagger$ the Moore-Penrose pseudoinverse of $\A$, then 
  \begin{align*}
    \| \w\|_2 &= \| \mathrm{Proj}_{\ker(\A)} (\x_0 - \x^*) \|_2 = \| \A^\dagger \A (\x_0 - \x^*) \|_2 \leq \| \A^\dagger \|_2 \|\A (\x_0 - \x^*) \|_2 \\
              &\leq 2 \epsilon \s_{\mathrm{min}}(\A) \leq 4 C^{-1} \left( 1 - \sqrt{\frac{m-1}{n}} \right)^{-1}\e\leq 8C^{-1}\e.
  \end{align*}
  holds with the probability on the right-hand side of \eqref{stab-subg}. Using this in \eqref{eq:C6-temp} with $p=2$ yields \eqref{eq:stab-cor-l2}. To show the $p=1$ result, we first use the inequality $\|\w\|_1 \leq \sqrt{n} \|\w\|_2$ in \eqref{eq:C6-temp}, which yields \eqref{eq:stab-cor-l2} with $p=1$. Note that this $\sqrt{n}$ factor is sharp when $\A$ is a Gaussian random matrix, cf. Theorem \ref{T1}.
\end{proof}

\textcolor{black}{Note that both bounds \eqref{case-relative} and \eqref{eq:stab-cor-l2} are comparable if one can choose $\alpha \lesssim \epsilon^2$, but this unfortunately makes \eqref{eq:stab-sample} quite restrictive. 
Similar issue also rises in the analysis of the $\ell_1$ minimization in Theorem \ref{T:l1}. 
Note also that \eqref{eq:stab-cor-l2} is consistent with the result in Theorem \ref{zhang-stability}.}

Now we provide the proof of Theorem \ref{T6}:   

\begin{proof}[Proof of Theorem \ref{T6}]
If both conditions \eqref{eq:cond} hold, then 
\begin{align*}
\|\x^*-\x_0\|^2_2&=\|\x^*\|_2^2+\|\x_0\|_2^2-2\langle\x^*, \x_0\rangle\\
&\leq (1+(1+\alpha)^2)\|\x_0\|_2^2-2\left(1-\frac{\alpha^2}{2}\right)\|\x_0\|^2_2
\leq 4\alpha\|\x_0\|^2_2, 
\end{align*}
where the last inequality uses $\alpha^2<\alpha$ for $0<\alpha<1$. Thus, we now restrict our attention to when \eqref{eq:cond} do not hold simultaneously; the goal in this case is to compute upper and lower bounds for $\|\x^*-\x_0\|_1/\|\x^*-\x_0\|_2$, with the intuition for the upper bound being motivated by \eqref{key}. In the sequel, we denote the violation of \eqref{eq:cond-1} or \eqref{eq:cond-2} as $\eqref{eq:cond-1}^\complement$ and $\eqref{eq:cond-2}^\complement$, respectively. When \eqref{eq:cond} is violated, our desired upper bound will take the form
\begin{align}\label{eq:l12-upper}
  \frac{\|\x^*-\x_0\|_1}{\|\x^*-\x_0\|_2} \leq \frac{4 \sqrt{s}}{\alpha}.
\end{align}
To show this, suppose first that \eqref{eq:cond-1} is violated, then 
\begin{align*}
  \frac{\|\x^*-\x_0\|_1}{\|\x^*-\x_0\|_2} &\leq \frac{\|\x^*\|_1+\|\x_0\|_1}{\sqrt{\|\x^*\|_2^2+\|\x_0\|_2^2 - 2 \langle \x^*, \x_0 \rangle}}
  \stackrel{\eqref{eq:cond-1}^\complement} \leq\frac{\|\x^*\|_1+\|\x_0\|_1}{\sqrt{\frac{\alpha^2}{2}(\|\x^*\|_2^2+\|\x_0\|_2^2)}} \\
    &\leq \frac{2}{\alpha}\cdot\frac{\|\x^*\|_1+\|\x_0\|_1}{\|\x^*\|_2+\|\x_0\|_2}
    \leq\frac{2}{\alpha}\cdot\frac{\|\x_0\|_1}{\|\x_0\|_2}
    \stackrel{\eqref{key}}{\leq} \frac{2}{\alpha}\sqrt{s} \leq \frac{4 \sqrt{s}}{\alpha},
\end{align*}
which is the desired inequality \eqref{eq:l12-upper}. A violation of the lower condition in \eqref{eq:cond-2} in addition to \eqref{key} implies $\|\x^*\|_1\leq \|\x_0\|_1$. We therefore split the case when \eqref{eq:cond-2} is violated into two dichotomous sub-cases.
\begin{enumerate}
  \item $\eqref{eq:cond-2}^\complement$, case 1: $\|\x^*\|_1\leq \|\x_0\|_1.$
    Let $S$ be the support of $\x_0$. 
    The following inequality is true:
    \begin{align*}
      \|\x_0\|_1 \geq \| \x^*\|_1 = \|\x_0+(\x^*-\x_0)\|_1&=\|(\x_0+(\x^*-\x_0))_S\|_1+\|\x^*_{S^\complement}\|_1\\
&\geq \|\x_0\|_1-\|(\x^*-\x_0)_S\|_1+\|\x^*_{S^\complement}\|_1,
    \end{align*}
    so that we must have $\|\x^*_{S^\complement}\|_1 \leq \|(\x^*-\x_0)_S\|_1$.  Therefore, the Cauchy-Schwarz inequality implies
    \begin{align*}
      \frac{\|\x^*-\x_0\|_1}{\|\x^*-\x_0\|_2}&\leq\sqrt{2}\cdot\frac{\|(\x^*-\x_0)_S\|_1+\|\x^*_{S^\complement}\|_1}{\|(\x^*-\x_0)_S\|_2+\|\x^*_{S^\complement}\|_2}\\
      &\leq \sqrt{8}\cdot\frac{\|(\x^*-\x_0)_S\|_1}{\|(\x^*-\x_0)_S\|_2}\\
      &\leq \sqrt{8s} \leq \frac{4\sqrt{s}}{\alpha},
    \end{align*}
    which is the desired inequality \eqref{eq:l12-upper}.
  \item $\eqref{eq:cond-2}^\complement$, case 2: $\|\x^*\|_1 > \|\x_0\|_1$ and $\|\x^*\|_2> (1+\alpha)\|\x_0\|_2$.
    The condition $\|\x^*\|_1>\|\x_0\|_1$ and \eqref{key} together imply $\|\x^*\|_2>\|\x_0\|_2$. With this, we have
    \begin{align*}
    \frac{\|\x^*-\x_0\|_1}{\|\x^*-\x_0\|_2}&\leq\frac{\|\x_0\|_1+\|\x^*\|_1}{\|\x^*\|_2-\|\x_0\|_2}\\
    &\leq\frac{\alpha+1}{\alpha}\cdot\frac{\|\x_0\|_1+\|\x^*\|_1}{\|\x^*\|_2}\\
    &\leq\frac{\alpha+1}{\alpha}\left(\frac{\|\x_0\|_1}{\|\x_0\|_2}+\frac{\|\x^*\|_1}{\|\x^*\|_2}\right)\\
    &\leq\left(2+\frac{2}{\alpha}\right)\sqrt{s} \leq 4 \frac{\sqrt{s}}{\alpha}, 
    \end{align*}
    where the last inequality uses the fact that $\alpha < 1$.
\end{enumerate}
When \eqref{eq:cond} is violated, we have established \eqref{eq:l12-upper}. We now begin the main part of the proof: decompose $\x = \u + \w$ as in the assumption. If $\|\u\|_p \leq \|\w\|_p$ for either $p = 1, 2$, then 
\begin{align*}
  \| \x_0 - \x^* \|_p \leq \| \u\|_p + \|\w\|_p \leq 2 \|\w\|_p \leq \frac{2-\beta}{1-\beta} \|\w\|_p,
 \end{align*}
for any $\beta \in (0,1)$, proving \eqref{stab-2-est}. Therefore, we can assume $\|\u\|_p > \|\w\|_p$ for both $p = 1, 2$. In this case, 
\begin{align*}
\frac{\|\x^*-\x_0\|_1}{\|\x^*-\x_0\|_2}\geq\frac{\|\u\|_1-\|\w\|_1}{\sqrt{\|\u\|_2^2+\|\w\|_2^2}}\geq\frac{1-\frac{\|\w\|_1}{\|\u\|_1}}{\sqrt{1+\frac{\|\w\|_2^2}{\|\u\|_2^2}}}\cdot\frac{\|\u\|_1}{\|\u\|_2}\geq h(v)\,\frac{\|\u\|_1}{\|\u\|_2}, 
\end{align*} 
where 
\begin{align}\label{eq:v-def}
  h(v) &\coloneqq \frac{1-v}{\sqrt{1 + v^2}}, & v&\coloneqq\max\left\{\frac{\|\w\|_1}{\|\u\|_1}, \frac{\|\w\|_2}{\|\u\|_2}\right\}.
\end{align}
If $h(v)\, \frac{\|\u\|_1}{\|\u\|_2}>\frac{4\sqrt{s}}{\alpha}$, then this contradicts \eqref{eq:l12-upper}, so that \eqref{eq:cond} must hold, showing \eqref{case-relative}. Therefore it only remains to consider when $h(v)\,\frac{\|\u\|_1}{\|\u\|_2}\leq\frac{4\sqrt{s}}{\alpha}$. Since $h(v) \geq \frac{1}{\sqrt{2}} (1-v)$ for $v > 0$, then 
\begin{align*}
  \frac{\beta}{\alpha \sqrt{2}} = \frac{4 \sqrt{s}}{\alpha} \frac{\|\u\|_2}{\|\u\|_1} \geq h(v) \geq \frac{1}{\sqrt{2}} (1-v),
\end{align*}
showing that $v \geq 1 - \frac{\beta}{\alpha}$.
Thus, if $p = 1$ or $2$ corresponds to whichever norm maximizes \eqref{eq:v-def}, then
\begin{align*}
  \|\u\|_p = v^{-1}\|\w\|_p \leq \left(1 - \beta/\alpha\right)^{-1} \|\w\|_p 
\end{align*}
If we add $\|\w\|_p$ to both sides and apply the triangle inequality to the left-hand side, we obtain the desired result \eqref{stab-2-est}.
\end{proof}

\begin{Rem}
The discussion above splits into two cases based on the conditions $\langle \x_0, \x^*\rangle \geq (1-\alpha^2/2)\|\x_0\|_2\|\x^*\|_2$ and $\|\x_0\|_2\leq\|\x^*\|_2\leq (1+\alpha)\|\x_0\|_2$. This would be unnecessary if one can show that $\|\x^*-\x_0\|_1/\|\x^*-\x_0\|_2$ is bounded by some universal constant (independent of $n$ and $\x^*$) times $\sqrt{s}$. Even though there is strong intuition that this is correct, a proof is currently elusive. 
\end{Rem}

\section{Numerical experiments}\label{numerical experiments}

In this section, we present several numerical simulations to complement our previous theoretical investigation. All the results in this section are repeatable on a standard laptop installed with $\texttt{R}$ \cite{RRR}. For simplicity, we will restrict to the noiseless case. The noisy case can be carried out similarly by tuning the parameter of the penalty term arising from the constraint. Since $\ell_1/\ell_2$ minimization problems are non-convex, our algorithms solve the problem approximately. \textcolor{black}{In particular, we utilize the algorithms from \cite{Wang_2020}, which is essentially the Inverse Power Method \cite{hein2010inverse} with an extra augmented quadratic term in $\x$-update.} For completeness, we briefly explain how the algorithm works: It was observed in \cite{Wang_2020} that subject to $\A\x=\b$, minimizing $\|\x\|_1/\|\x\|_2$ is equivalent to minimizing $\|\x\|_1-\alpha \|\x\|_2$ for some $\alpha\in [1,\sqrt{n}]$, where $\alpha$ is some case-dependent parameter. Since the true value of $\alpha$ is unknown, one can start with an initial guess and update it using a bisection search. At iteration $k$, a full bisection search requires solving a minimization problem of the form $\min_{\A\x=\b}\|\x\|_1-\alpha^{(k)}\|\x\|_2$ and the minimizer will be used to update $\alpha^{(k)}$. To accelerate, an adaptive algorithm based on the difference of convex functions algorithm (see \cite{Tao_1998}) was proposed in \cite{Wang_2020} by replacing $\alpha^{(k)}\|\x\|_2$ by its linearization at the previous iterate $\x^{(k-1)}$ with an additional regularization term scaled by a tunable parameter $\beta$. Given an initialization $\x^{(0)}$ and $\alpha^{(0)}$, the alternating algorithm can be summarized as follows:
\begin{align*}
\begin{cases}
\x^{(k+1)}&=\arg\min_{\A\x=\b}\left\{\|\x\|_1-\frac{\alpha^{(k)}}{\|\x^{(k)}\|_2}\langle \x, \x^{(k)}\rangle+\frac{\beta}{2}\|\x-\x^{(k)}\|_2^2\right\}\\
\alpha^{(k+1)}&=\frac{\|\x^{(k+1)}\|_1}{\|\x^{(k+1)}\|_2}.
\end{cases}
\end{align*}
The $\x$-subproblem can be efficiently solved using the Alternating Direction Method of Multipliers (ADMM), see \cite{Boyd_2010}. It was shown in \cite{Wang_2020} that small regularization parameters $\beta$ tend to yield better local decay rate. In our simulations we use $\beta=0.5$. A choice for the penalty parameter $\rho$ in the ADMM algorithm (not shown above) is slightly more tricky. In our experiment we choose $\rho=20$, but other kinds of simulations may require tuning for a different value of $\rho$. We emphasize that the lack of certainty about the choice of these parameters is a drawback of the algorithms in general, and is not introduced by our implementation or choice of application. Since ADMM is a relaxation scheme, it can only approximately solve the original problem, resulting in a solution vector upon termination many of whose components have small magnitude. To increase the stability of the algorithm, a box constraint based on prior information will be incorporated, and the details will be specified later.   

\subsection{Initialization and support selection}
A common issue in solving non-convex optimization problems is that algorithms may become trapped at local minimizers. This phenomenon is particularly worrying in our case. Indeed, due to the scale-invariant structure, the objective function $\ell_1/\ell_2$ may have infinitely many local minimizers in the feasible set. As a result, global convergence of the above algorithm depends on a good initialization. A natural choice would be the $\ell_1$ minimizer or the first a few steps of the iterative reweighted least squares (IRLS) solving $\ell_0$, see \cite{foucart2009sparsest}, \cite{Daubechies_2008} and \cite{Lai_2013}. The intuition of these choices is that the $\ell_1$ (or $\ell_q$ if IRLS is used) minimizer is not too far from $\x_0$, and $\x_0$ is one of the minimizers of $\ell_1/\ell_2$. Therefore, success of the above algorithm under such initialization heavily relies on the `approximate' success of the $\ell_1$ minimization. This observation will be numerically verified later. To overcome the strong dependence on the $\ell_1$ ($\ell_q$) minimization, we will propose a novel initialization approach based on a support selection process to make the algorithm less reliant on the $\ell_1$ minimizer and leads to improved results.

We propose to initialize $\x^{(0)}$ in a way that utilizes the information of the support of $\x_0$. Unfortunately, the support of $\x_0$ is generally unknown. As a substitute, one can use the support of the recovered solution from other algorithms. Here we will interpret the support of the $\ell_1$ minimizer as near-oracle identification of the support of $\x_0$. Indeed the theoretical uniform recoverability of $\ell_1$ minimization makes it superior to most greedy algorithms, and algorithmic implementations of $\ell_1$ minimization have better convergence guarantees than many other non-convex algorithms. For a fixed sparsity $s$, we first compute the best $s$-term approximation of the $\ell_1$ minimizer, which we denote by $\x_{s}$. Instead of using $\x_{s}$ directly for initialization, we will consider each element of the support separately. For every $i\in\text{supp}(\x_{s})$, we consider the initialization $\x^{(0)}$ as defined by a vector whose $i$-th component is $\langle \b,\mathbf{a}_i\rangle/\|\mathbf{a}_i\|_2^2$ and the other components are $0$, where $\mathbf{a}_i$ is the $i$-th column vector of $\A$. The idea behind this is to counteract the influence of incorrectly detected components in the support on the correctly detected support. 
One may also view this as a way to mitigate algorithm failure (see \cite{Rahimi_2019}) via multi-initialization. In total, one needs to solve $s$ subproblems of $\ell_1/\ell_2$ minimization and choose the solution that gives the smallest $\ell_1/\ell_2$ value. Thus, the computational complexity will be $s$ times more than that of solving a single $\ell_1/\ell_2$ minimization problem. This increased cost can be mitigated via parallel computing since the subproblems are embarrassingly parallel. We will call the $\ell_1/\ell_2$ algorithm with this proposed initialization ``$\ell_1/\ell_2$+SS", where SS stands for \emph{support selection}. In fact, this multi-initialization approach is also applicable to other iterative algorithms. 

\subsection{$\ell_1 / \ell_2$ simulation particulars}
In the simulation below, we choose the measurement matrix $\A$ to be a $50\times 250$ Gaussian random matrix with iid standard normal entries. $s$ is the sparsity level of generated vectors $\x_0$. For each fixed value of $s$, we generate $\x_0$ by randomly choosing $s$ of its components to be nonzero. The nonzero components are independently drawn from distributions with different dynamic range: $\text{Uniform}([-10, 10])$ and $\text{Uniform}([-10, 5]\cup [5, 10])$.

An additional box constraint $\|\x\|_\infty\leq 10$ based on this prior information on magnitude of entries is computationally imposed during iterations to solve the $\ell_1/\ell_2$ problem in both cases. Note that this extra constraint does not change the problem if $\x_0$ is the global minimizer. If $\x_0$ is not the global minimizer, the box constraint will disallow solution vectors with erroneous large magnitude, making the algorithm more stable in practice.

To detect exact recovery, we will use a slightly different criteria than the commonly used relative error threshold in other literature. We say that a computed solution $\x$ recovers $\x_0$ if the support of the best $50$-term approximation of $\x$ under the $\ell_1$ norm contains the support of $\x_0$. Indeed, in this case $\x_0$ can be easily reconstructed by solving $\A\x=\b$ with $\A$'s columns restricted to the support of the $50$ largest components of $\x$. The reason for this criteria is due to computational reasons. Based on our choice of regularization and relaxation parameters, the computational error of $\ell_1/\ell_2$ cannot be made as small as many other algorithms with known convergence guarantees.

\subsection{List of algorithms}
Now we compare the $\ell_1/\ell_2$+SS algorithm ({l1/l2+SS}) described above with the box constraint against the following popular non-convex (and $\ell_1$) methods in sparse recovery: 
\begin{itemize}
  \item $\ell_1$ minimization ({l1}): The box constraint is included for consistency in comparison. We will use the linear programming package \texttt{lpSolve} (\cite{LP}) in \texttt{R} to solve it.

\medskip

\item Reweighted $\ell_1$ minimization ({RWl1+l1}): The box constraint is included for consistency in comparison. We will use the algorithm in Section 2.2 of \cite{Candes_2008}  to solve it, with the regularization parameter $\e=0.1$. The initialization is set as the $\ell_1$ minimizer.

\medskip 

\item $\ell_{1/2}+\ell_1$ minimization ({l1/2+l1}): We will use the IRLS algorithm in \cite{foucart2009sparsest} to solve it. (We do not use the improved versions in \cite{Daubechies_2008} or \cite{Lai_2013} as they require knowledge on the NSP/RIP of $\A$, which is hard to compute in practice. For technical reasons, we are not able to incorporate the box constraint in this case.) The initialization is set as the $\ell_1$ minimizer.

\medskip

\item $\ell_{1/2}$+SS minimization ({l1/2+SS}): This is the same as the above $\ell_{1/2}$ algorithm but initialized with the additional support selection process introduced above.

\medskip

\item $\ell_1-\ell_2$ minimization ({l1-l2+l1}): The box constraint is included for consistency in comparison. We will use the algorithm in \cite{Yin_2015} to solve it. The Lasso penalty parameter and the ADMM penalty parameter are chosen to be $0.01$ and $100$, respectively. The stopping rules are the same as the one proposed in \cite{Yin_2015}. The initialization is set as the $\ell_1$ minimizer. 

\medskip

\item $\ell_1/\ell_2+\ell_1$ minimization ({l1/l2+l1}): The box constraint is included for consistency in comparison. We use the adaptive algorithm in \cite{Wang_2020} to solve it with the $\ell_1$ initialization. 

\medskip

\item \textcolor{black}{Orthogonal Marching Pursuit ({OMP}): We use the OMP algorithm in \cite{Tropp_2007}. The stopping criterion is either the length of the residual falls below $10^{-8}$ or the size of the detected support exceeds the total number of measurements, which in our case equals $50$. }

\medskip

\item \textcolor{black}{Compressive Sampling Marching Pursuit ({CoSaMP}): We use the CoSaMP algorithm in \cite{Needell_2010}. The stopping criterion is either the length of the residual falls below $10^{-8}$ or the total iteration step exceeds $100$. The use of the maximal iteration step in the halting rule is due to that convergence of the CoSaMP is guaranteed if the measurement matrix satisfies certain RIP conditions \cite{Satpathi_2017}. When sparsity gets larger, the required RIP condition is no longer valid thus the algorithm may not converge.}

\end{itemize}

The sparsity $s$ ranges from $6$ to $24$ by increments of 2, and for each $s$ we perform $100$ independent experiments with the average recovery rates and relative error recorded. The complete average comparison simulation is run $100$ times with the quantiles plotted to demonstrate the uncertainty in the simulation. The quantile levels are chosen as $0.2$-$0.5$-$0.8$ for each method. The results are given below.
  
\begin{figure}[htbp]
\centering
  \includegraphics[width=0.45\textwidth]{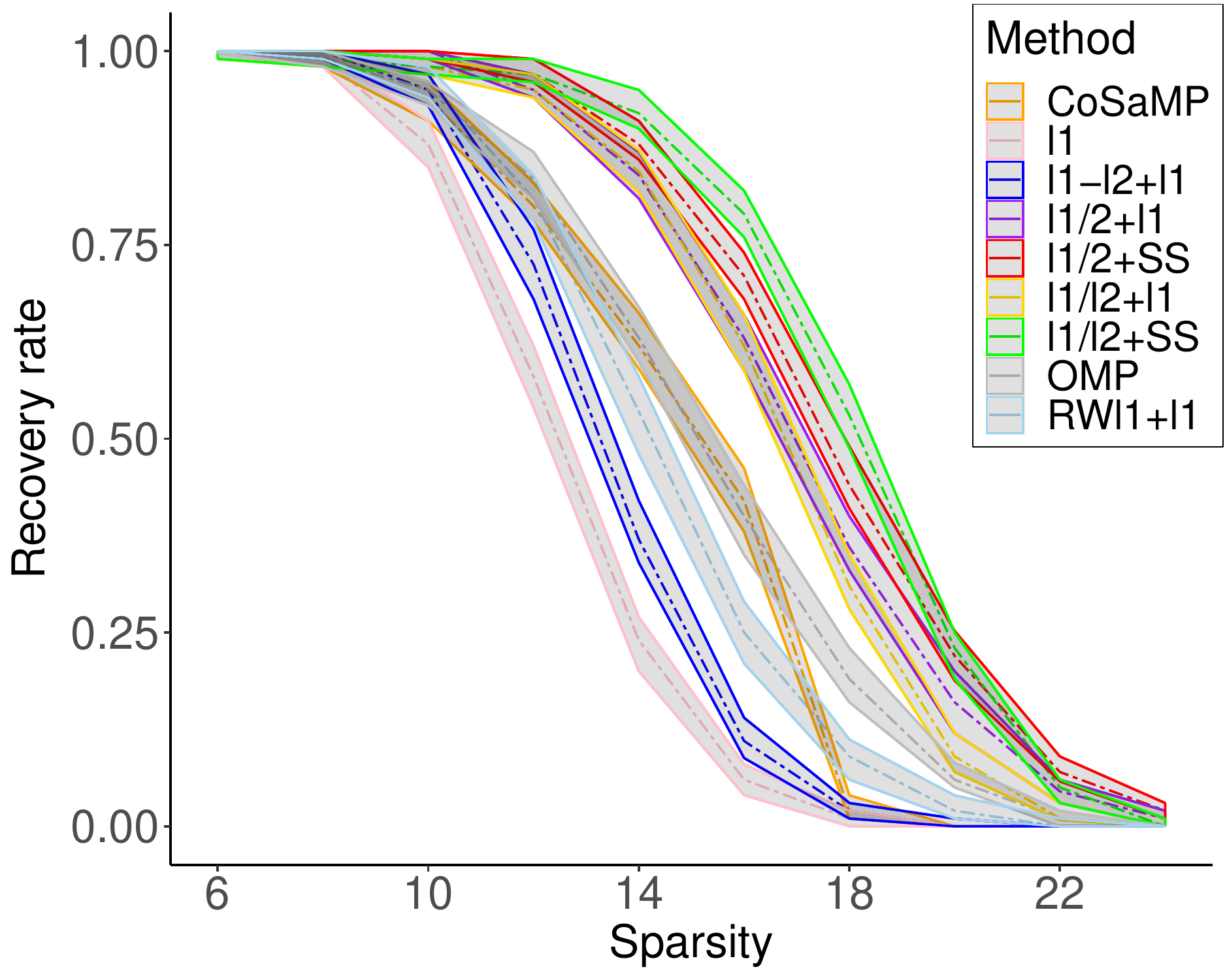}\ \ \ \ 
  \includegraphics[width=0.45\textwidth]{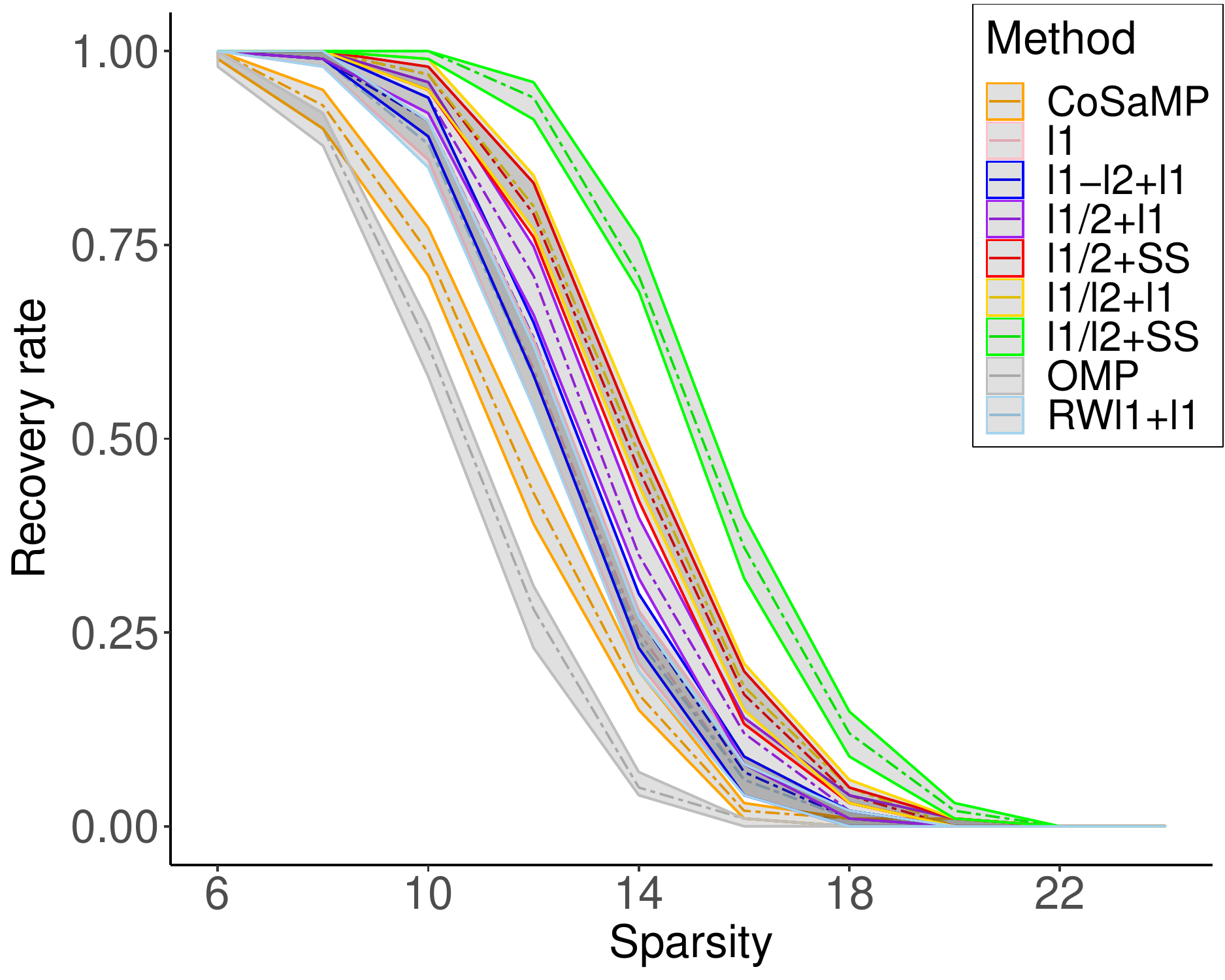}
  \caption{$0.2$-$0.5$-$0.8$ quantile band for the average recovery rate: Left: Uniform$([-10, 10])$ coefficients. Right: Uniform$([-10, -5]\cup [5, 10])$ coefficients.}\label{avg-rate}
\end{figure}

\textcolor{black}{To compare the scalability of the algorithms, we compute the average processing time of each algorithm applied to signals of varying length. The range of length of the signal $n$ is chosen between $2^6$ and $2^{12}$, increasing by a multiple constant $2$ at a time. 
The number of measurements $m$ is set as $m=n/4$, and the sparsity level $s$ is set as $s=m/4$. The components in the support of the signal are uniformly generated from $[-10,10]$. 
For each algorithm, its average processing time is computed as the average time of running $10$ independent samples. Since the support selection procedure is algorithmically equivalent to applying a single-initialization algorithm $s$ times, its average processing time is taken as $s$ multiplied by the time for the same algorithm without support selection. The results are given in Figure \ref{time}. Many of the algorithms exhibit similar asymptotic computational complexity, although CoSaMP, {l1-l2+l1}, and {l1/l2+l1} have slightly better complexity. The computational cost for the support selection procedures is higher than most of the rest, but the asymptotic complexity is similar.} 

\begin{figure}[htbp]
\centering
  \includegraphics[width=0.45\textwidth]{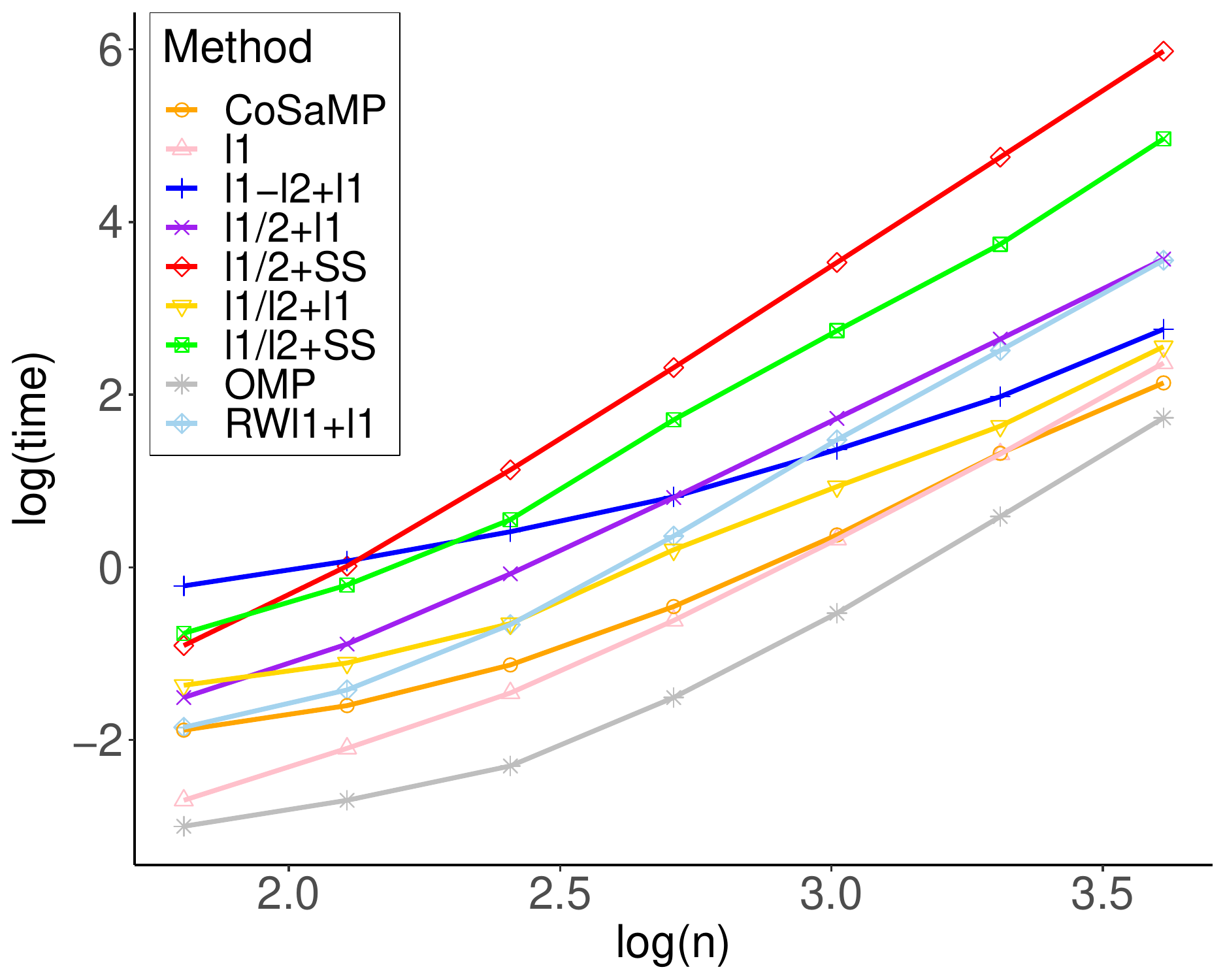}
  \caption{\textcolor{black}{Average processing time of the algorithms applied to signals of different length. Both the $x$ and $y$ axis are plotted using the logarithm with based $10$. }}\label{time}
\end{figure}

\subsection{Simulation results}

It can be seen from Figure \ref{avg-rate} that for both types of coefficients, $\ell_1/\ell_2$ with the box constraint and SS-initialization has the best performance among all non-convex optimization methods under comparison. $\ell_{1/2}$ also performs fairly well but is slightly inferior to $\ell_1/\ell_2$. This is no surprise since $\ell_1/\ell_2$ utilizes the box constraint which is absent in the $\ell_{1/2}$ algorithm. On the other hand, by taking the SS-initialization, the recovery rate of $\ell_1/\ell_2$ has significantly improved compared to a similar step taken for $\ell_q$. This implies that $\ell_1/\ell_2$ is more sensitive to the initial value and the multi-initialization step enhances the success rate of the algorithm. By comparing the left- and right-hand panels in Figure \ref{avg-rate}, it is easy to observe that all the methods under comparison perform better when the dynamic range of the coefficients is large. This phenomenon can be well explained for reweighted $\ell_1$ and $\ell_q$, in which a reweighting step is used to reduce the bias between $\ell_q$ ($0<q\leq 1$) and $\ell_0$. However, for $\ell_1/\ell_2$, this is not well understood. We provide some theoretical evidence in Theorem \ref{local optimality} for this behavior in terms of the local optimality condition; nevertheless, a complete understanding of this is still absent.

\textcolor{black}{It can be seen from Figure \ref{time} that based on our choice of algorithms, $\ell_1/\ell_2$ with single initialization demonstrates a reasonable computational time asymptotically. It is more expensive than the greedy algorithms and $\ell_1$, of which the solution is used to give a good initialization for $\ell_1/\ell_2$. It is almost at the same level as the $\ell_1-\ell_2$ since both used the ADMM relaxation in the computation. 
Meanwhile, it is cheaper than the other non-convex algorithms such as $\ell_q$ and reweighted $\ell_1$ which either require matrix inversion or solving a linear programming problem in each iteration (more advanced numerical methods can help accelerate computation in practice, but we do not investigate it here). The support selection procedure increases the processing time of the algorithms by a multiple factor of the sparsity level. When sparsity is large, this effect is not negligible but can be mitigated via parallel computing.}

It is worth pointing out that although $\ell_1/\ell_2$ algorithms yield better recovery results when the magnitude of the entries of $\x_0$ are known \emph{a priori} to be bounded from above, their recovery rate is closely related to the accuracy of the $\ell_{1}$ ($\ell_q$) minimizer. If the solution $\x$ obtained from minimizing $\ell_{1}$ ($\ell_q$) is incoherent with $\x_0$, then it is unlikely that $\ell_1/\ell_2$ will give substantially better result. Our initialization approach proposed earlier is not able to completely remove such a dependence, and only mitigates the impact. However, it is likely that the support of the $\ell_1$ minimizer contains at least one component that lies in the true support of $\x_0$. If one of these components happens to be close to the $\ell_1/\ell_2$ convergence regime of $\x_0$, then the support selection process will promote convergence to $\x_0$ by removing the influence of other elements in the detected support. Figure \ref{cor} below verifies this point.

\begin{figure}[htbp]
\centering
  \includegraphics[width=0.45\textwidth]{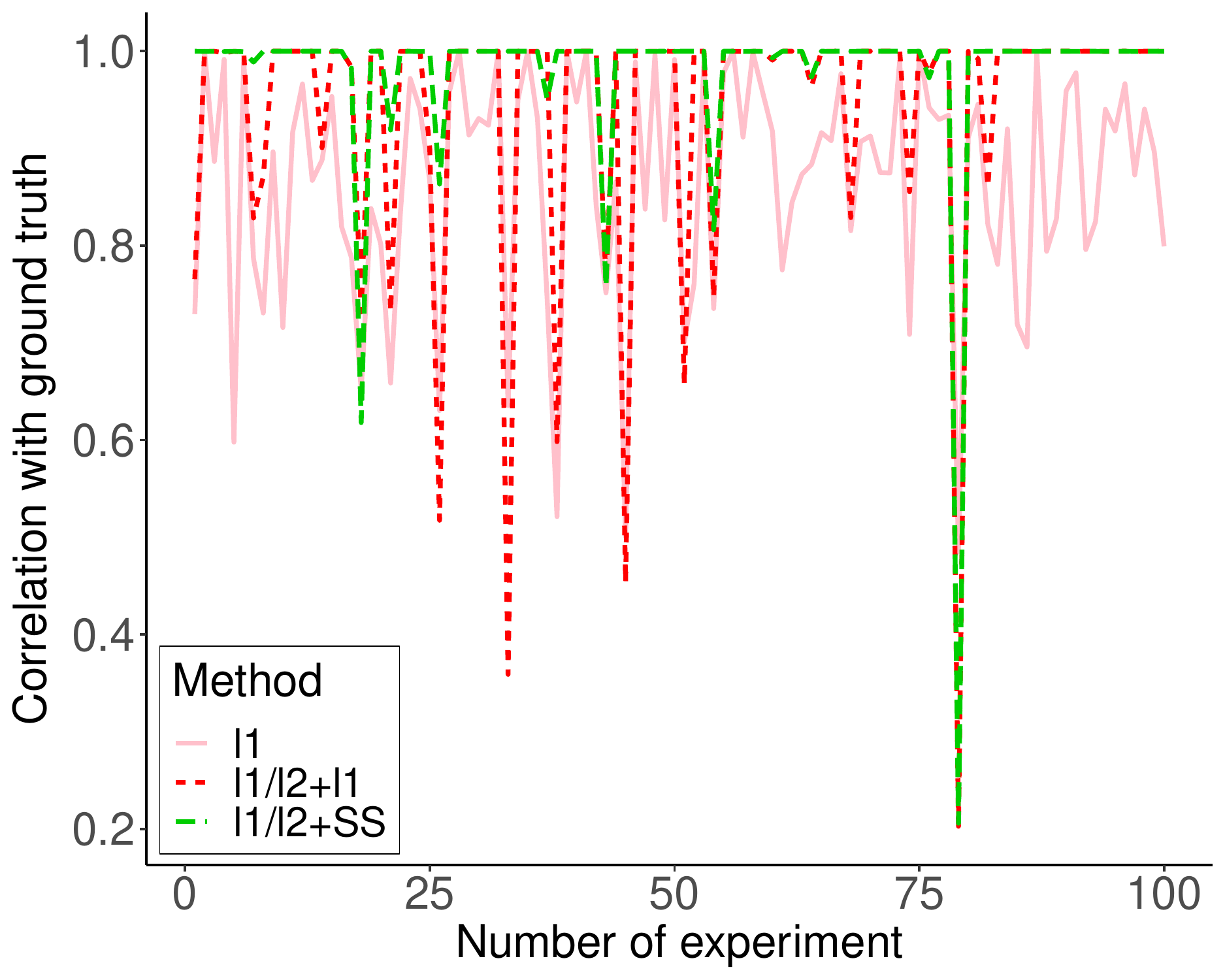}\ \ \ \ 
  \includegraphics[width=0.45\textwidth]{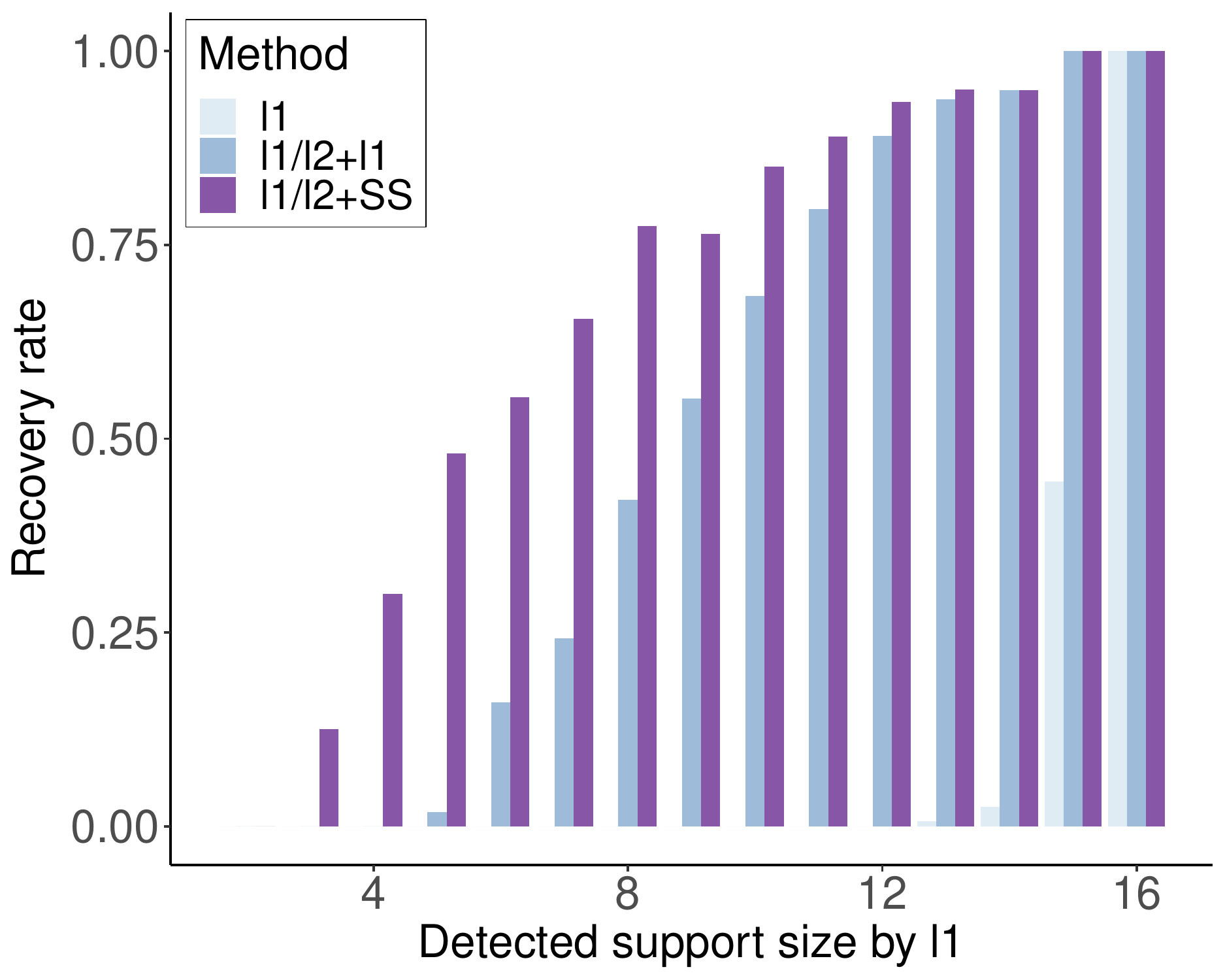}
  \caption{\textcolor{black}{Numerical experiments on the support selection initialization. 
  Left: Correlation between the minimizers of $\ell_{1}$, $\ell_1/\ell_2+\ell_1$, $\ell_1/\ell_2$+SS and the ground truth $\x_0$ in $100$ experiments in the case of coefficients drawn from the distribution Uniform$([-10, 10])$ with sparsity level $s=16$. Right: Recovery rate of $\ell_{1}$, $\ell_1/\ell_2+\ell_1$, and $\ell_1/\ell_2$+SS as a function of $\ell_1$-detected support size over $2000$ experiments (total), with Uniform$([-10, 10])$ coefficients and sparsity level $s=16$. }}
  \label{cor}
\end{figure}

In Figure \ref{cor}, the left panel illustrates the correlation between the minimizers of $\ell_{1}$, $\ell_1/\ell_2+\ell_1$, $\ell_1/\ell_2$+SS and the true signal $\x_0$ in $100$ experiments when $s=16$ and the coefficients are chosen from Uniform$[-10,-10]$. It can be seen that the general trend of the three curves is similar. When the correlation between the $\ell_1$ minimizer and $\x_0$ is low, say below $0.5$, it is also low for both the $\ell_1/\ell_2$ minimizers. This implies that success of the $\ell_1/\ell_2$ algorithms heavily relies on the $\ell_1$ minimizer being reasonably close to $\x_0$. 
\textcolor{black}{On the other hand, the right panel compares the recovery rate of the three methods for different sizes of the $\ell_1$-detected support over $2000$ experiments. 
The $\ell_1$-detected support refers to the number of indices in the support of $\x_0$ that are correctly detected by the best $s$-sparse truncation of the $\ell_1$ minimizer. 
When the detected support is close to the true support, there is a large chance that both $\ell_1/\ell_2$ algorithms can successfully push it towards $\x_0$.  
As the detected support size diminishes, $\ell_1/\ell_2$ with the support-selection based initialization appears to do a better job. 
This provides some numerical evidence that utilizing support information of the $\ell_1$ minimizer through support selection helps reduce algorithm failure. 
}


\textcolor{black}{In Figure \ref{case study}, we give a concrete example in which both $\ell_1$ and $\ell_1/\ell_2$ initialized with $\ell_1$ failed to recover $\x_0$ but with the additional support selection process, it succeeded. 
The left panel visualizes the structures of both the true and recovered signals. 
The support of the true signal is
$$\text{supp}(\x_0) = \{212, 194, 66, 73, 132, 248, 234, 70, 12, 249, 226, 102, 69, 85, 201, 106\},$$ 
arranged according to decreasing magnitude of the entries.
In particular, $\x_0[212] = -9.832$ and $\x_0[106]=0.158$. 
The $\ell_1$ minimization completely fails to recover the solution in the sense that $\text{supp}(\x_0)\not\subseteq\text{supp}(\x_{\ell_1})$, where $\x_{\ell_1}$ is the $\ell_1$ minimizer. 
In this case, the $\ell_1/\ell_2$ with $\ell_1$ initialization moves to a local minimizer near $\x_{\ell_1}$ which is different from $\x_0$, whereas the same algorithm with support selection initialization successfully detects $\x_0$. 
The right panel gives a more careful comparison between the true components and the recovered components using $\ell_1/\ell_2$+SS on the support.}

\textcolor{black}{To better understand the success of the support selection procedure, we compare the $\ell_1/\ell_2$ objective value of the $s$ minimizers obtained from initializing on each component in the support of the best $s$-approximation of $\x_{\ell_1}$. This is given in Table \ref{sstudy}. 
In our example, the support of the best $s$-approximation of $\x_{\ell_1}$ is $$\text{supp}(\x_{\ell_1}|_s) = \{234, 137, 212, 66, 145, 85, 87, 102, 194, 132, 99, 110, 205, 40, 246, 128\}.$$
The detected support by the $\ell_1$ is  
\begin{align*}
\text{supp}(\x_0)\cap\text{supp}(\x_{\ell_1}|_s) = \{212, 194, 66, 132, 234, 102, 85\}.
\end{align*}
The $\ell_1/\ell_2$ objective values of $\x_0$ and the solutions found by the $\ell_1$ and $\ell_1/\ell_2+\ell_1$ are $3.456, 5.173$ and $4.510$, respectively.  It is clear from Table \ref{sstudy} that support selection procedure initialized at indices $212, 194, 66$ and $132$ succeeded in recovering $\x_0$ (up to some computational error). 
These indices are in the support of $\x_0$. 
Note that initialization at other indices which are also in $\text{supp}(\x_0)$ such as $234, 102$ and $85$ results algorithm failure. A possible explanation for this is that the magnitude of $\x_0$ on these indices is relatively small compared to that on the indices leading to success. }

\begin{figure}[htbp]
\centering
  \includegraphics[width=0.45\textwidth]{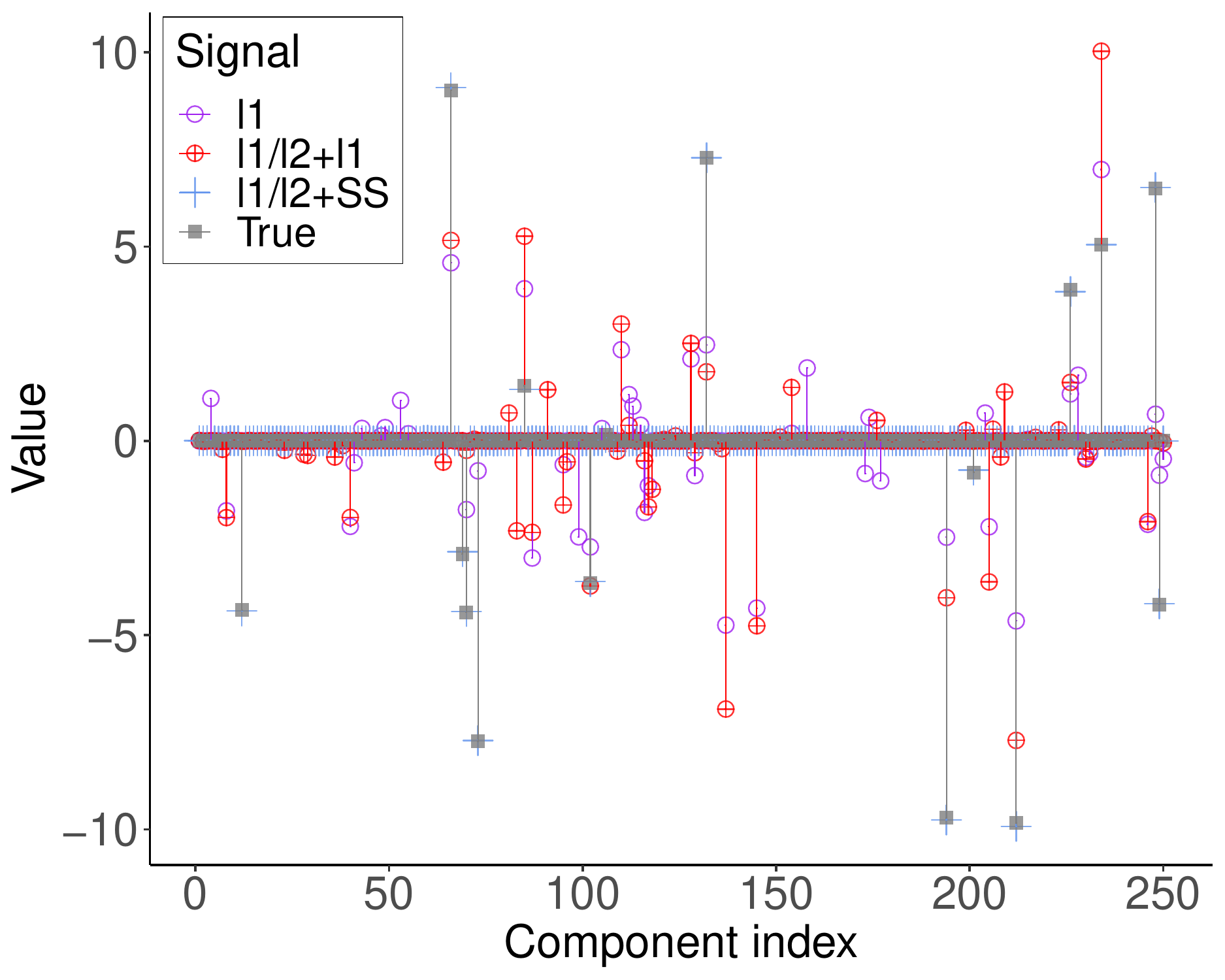}\ \ \ \ 
  \includegraphics[width=0.45\textwidth]{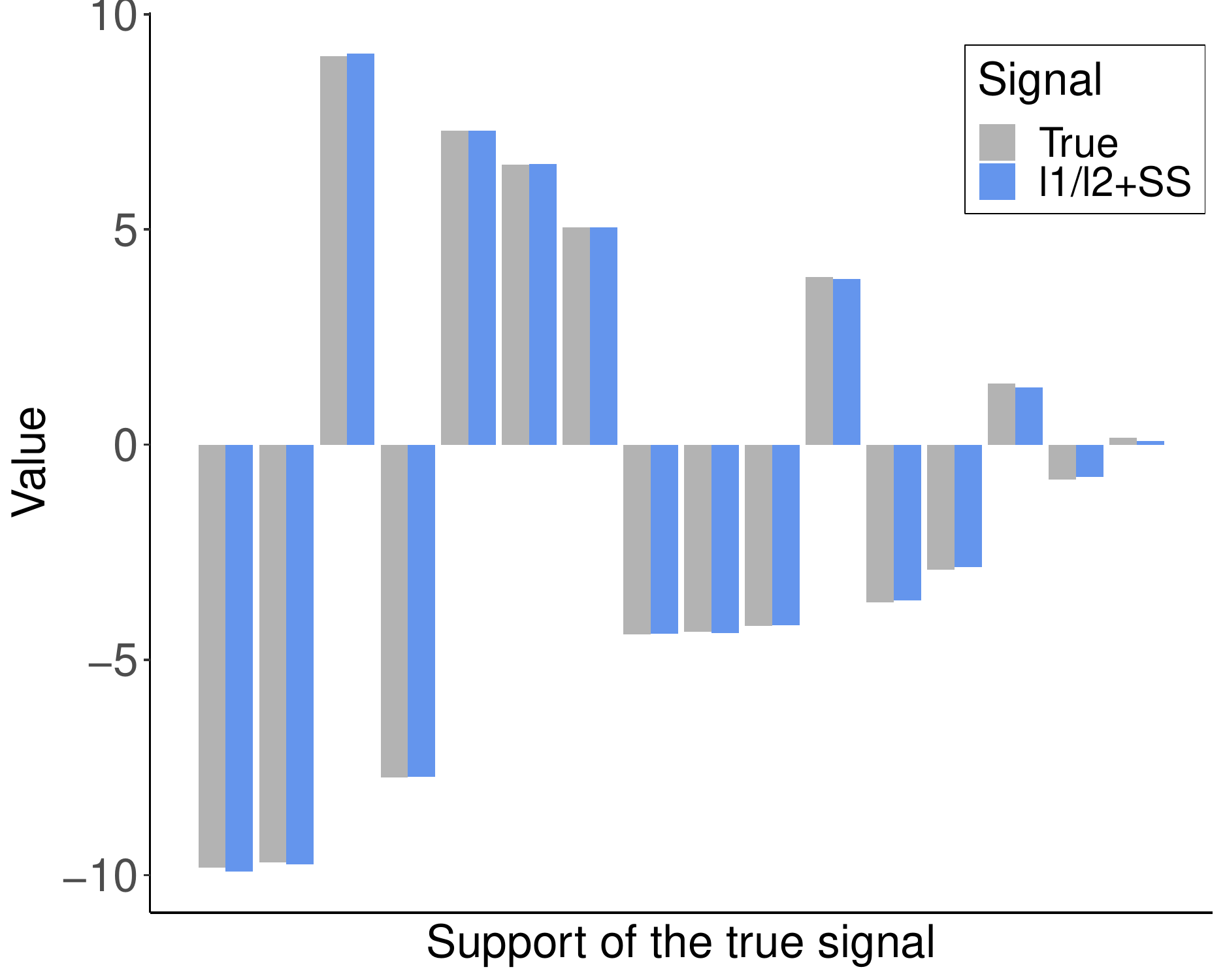}
  \caption{\textcolor{black}{A case where both $\ell_1$ and $\ell_1/\ell_2+\ell_1$ failed to recover $\x_0$ but $\ell_1/\ell_2$+SS succeeded in the case of coefficients Uniform$([-10, 10])$ and sparsity level $s=16$: Left: General distribution of magnitude of the true components and recovered components using $\ell_1$, $\ell_1/\ell_2+\ell_1$ and $\ell_1/\ell_2$+SS. Right: Careful comparison between the true components and the recovered components using $\ell_1/\ell_2$+SS on the support. 
  }}\label{case study}
\end{figure}

\begin{table}
\small
\begin{center}
\begin{tabular}{l*{6}{c}r}
Initialized index in $\text{supp}(\x_{\ell_1}|_s)$       &  $\ell_1/\ell_2$  \\
\hline
234 & 4.235  \\
137     & 4.450 \\
 \textbf{212}   & \textbf{\textcolor{red}{3.483}} \\
\textbf{66}    & \textbf{\textcolor{red}{3.483}} \\
145     & 4.401 \\
 85    & 4.277 \\
87     & 4.282 \\
102     & 4.248 \\
\textbf{194}     &\textbf{\textcolor{red}{3.483}} \\
\textbf{132}     & \textbf{\textcolor{red}{3.483}} \\
99   & 4.520 \\
110     & 4.361 \\
205     & 4.493 \\
40     & 4.505 \\
246     & 4.520 \\
128    & 4.480 \\
\end{tabular}
\end{center}
 \caption{\textcolor{black}{Comparison of the $\ell_1/\ell_2$ objective value from the $16$ single-support initializations in the Figure \ref{case study}. Initializations started from support indices $212, 66, 194$ and $132$ lead to solutions with the best $\ell_1/\ell_2$ value $3.483$, which approximately matches the $\ell_1/\ell_2$ value ($3.456$) of $\x_0$. }}\label{sstudy}
\end{table}

\section{Conclusion and future work}\label{sec:con}
We have theoretically and numerically investigated the $\ell_1/\ell_2$ minimization problem in the context of recovery of sparse signals from a small number of measurements. We have provided a novel local optimality criterion in Theorem \ref{T:local}, which gives some theoretical justification to the empirical observation that $\ell_1/\ell_2$ performs better when the nonzero entries of the sparse solution have a large dynamic range. We also provide a uniform recoverability condition in Theorem \ref{T5} for the $\ell_1/\ell_2$ minimization problem. Our final theoretical contribution is a robustness result in Theorem \ref{T6} that can be used to provide stability for noisy $\ell_1/\ell_2$ minimization problems, see Corollary \ref{cor:C6}. We have also proposed a new type of initialization for this nonconvex optimization problem called \emph{support selection} that empirically improves the recovery rate for $\ell_1/\ell_2$ minimization. Investigations that give a better theoretical understanding of why large dynamic range improves this type of minimization, along with additional analysis to better quantify stability in noisy cases, will be the subject of future research.

Although our analysis in this article arrives in similar recoverability and stability conditions analogous to the ones given by $\ell_1$, it does not give anything better. This may be due to the fact that the inequalities originally sharp for $\ell_1$ become less optimal when additional division steps are taken in the estimates. Also, norm ratios are more of objectives promoting the compressibility of a signal rather than the sparsity defined by $\ell_0$, which is highly discontinuous. Whereas in many practical problems, a sparse signal comes from the approximation of a compressible signal. This suggests that $\ell_1/\ell_2$ itself could be an alternative objective in terms of defining the goal of compressed sensing. More theoretical work in this direction is also worth exploring in the future. 

\section*{Conflict of interest}

There is none. 

\section*{Acknowledgements}
We would like to thank the anonymous referees for their very helpful comments which significantly improve the presentation of the paper. 
The first author (yxu@utah.math.edu) thanks for useful discussions with Tom Alberts, You-Cheng Chou and Dong Wang.  The first and second authors (yxu@math.utah.edu, akil@sci.utah.edu) acknowledge partial support from NSF DMS-1848508.
The third author (tranha@ornl.gov) acknowledges support from Scientific Discovery through Advanced Computing (SciDAC) program through the FASTMath Institute under Contract No. DE-AC02-05CH11231.
The last author (cwebst13@utk.edu) acknowledges the U.S. Department of Energy, Office of Science, Early Career Research Program under award number ERKJ314; U.S. Department of Energy, Office of Advanced Scientific Computing Research under award numbers ERKJ331 and ERKJ345; and the National Science Foundation, Division of Mathematical Sciences, Computational Mathematics program under contract number DMS1620280.

\printbibliography
  
\end{document}